\documentclass[preprint,12pt]{elsarticle}

\usepackage{graphicx}
\usepackage{amsfonts}
\usepackage{amssymb}
\usepackage{amsmath}
\usepackage[greek,spanish,english]{babel}
\usepackage{eurosym}
\usepackage{epic}
\usepackage{eepic}
\usepackage{srcltx} 
\usepackage{hhline}
\usepackage{multirow}
\usepackage{color}
\usepackage{lscape}
\usepackage{fancyhdr}
\usepackage{dsfont}
\usepackage{mathrsfs}
\usepackage{latexsym}
\usepackage{epsfig}
\usepackage{multicol}
\usepackage{geometry}
\usepackage{hyperref}
\usepackage{stmaryrd}
\usepackage{wrapfig}
\usepackage{pifont}
\usepackage{bbding}
\usepackage{epstopdf}
\usepackage{amsthm}

\newtheorem{theorem}{Theorem}

\newtheorem{lemma}{Lemma}

\begin{document}

\begin{frontmatter}

\title{Multipoint flux mixed finite element methods for slightly compressible flow in porous media}


\author[a]{A. Arrar\'as\corref{cor1}}
\ead{andres.arraras@unavarra.es}
\author[a]{L. Portero}
\ead{laura.portero@unavarra.es}
\cortext[cor1]{Corresponding author.}
\address[a]{Departamento de Ingenier\'{\i}a Matem\'atica e Inform\'atica, Universidad P\'ublica de Navarra, Edificio de Las Encinas, Campus de Arrosad\'{\i}a, 31006 Pamplona (Spain)\\[-0.7ex]}

\begin{abstract}
In this paper, we consider multipoint flux mixed finite element discretizations for slightly compressible Darcy flow in porous media. The methods are formulated on general meshes composed of triangles, quadrilaterals, tetrahedra or hexahedra. An inexact Newton method that allows for local velocity elimination is proposed for the solution of the nonlinear fully discrete scheme. We derive optimal error estimates for both the scalar and vector unknowns in the semidiscrete formulation. Numerical examples illustrate the convergence behaviour of the methods, and their performance on test problems including permeability coefficients with increasing heterogeneity.
\end{abstract}


\end{frontmatter}


\section{Introduction}

Slightly compressible single-phase Darcy flow in porous media is governed by the following nonlinear parabolic initial-boundary value problem \cite{bea:88}
\begin{subequations}\label{ibvp}
	\begin{align}
	&\mathbf{u}=-K\rho(p)\left(\nabla p-\rho(p)\mathbf{g}\right)&&\hspace*{-2cm}\mbox{in\,\,}\Omega\times(0,T],\label{ibvp:a}\\[0.5ex]
	&\phi\,\rho(p)_t+\nabla\cdot\mathbf{u}=f&&\hspace*{-2cm}\mbox{in\,\,}\Omega\times(0,T],\label{ibvp:b}\\[0.5ex]
	&p=0&&\hspace*{-2cm}\mbox{on\,\,}\Gamma_D\times(0,T],\label{ibvp:c}\\
	&\mathbf{u}\cdot\mathbf{n}=0&&\hspace*{-2cm}\mbox{on\,\,}\Gamma_N\times(0,T],\label{ibvp:d}\\
	&p=p_0&&\hspace*{-2cm}\mbox{in\,\,}\Omega\times\{0\},\label{ibvp:e}
	\end{align}
\end{subequations}
where $\Omega\subset\mathbb{R}^d$, $d=2$ or $3$, is a convex polygonal or polyhedral domain with Lipschitz continuous boundary given by $\partial\Omega=\overline{\Gamma}_D\cup\overline{\Gamma}_N$ such that $\Gamma_D\cap\Gamma_N=\emptyset$. In this formulation, the equations (\ref{ibvp:a}) and (\ref{ibvp:b}) represent Darcy's law and the conservation of mass, respectively. The unknowns are the fluid pressure $p(\mathbf{x},t)$ and the Darcy velocity $\mathbf{u}(\mathbf{x},t)$. Additional data include the porosity of the medium $\phi(\mathbf{x})$, the mass flux source term $f(\mathbf{x},t)$, a second-order symmetric and positive definite tensor $K(\mathbf{x})$ representing the rock permeability divided by the fluid viscosity, and the gravitational acceleration vector $\mathbf{g}$. The vector $\mathbf{n}$ denotes the outward unit normal on $\partial\Omega$. The model is closed with an equation of state that determines the fluid density $\rho$ as a function of the pressure $p$, i.e.,
\begin{equation*}
\rho(p)=\rho_{\mathrm{ref}}\,e^{c_f(p-p_{\mathrm{ref}})},
\end{equation*}
where $\rho_{\mathrm{ref}}$ and $p_{\mathrm{ref}}$ are the reference density and pressure, respectively, and $c_f$ is the fluid compressibility constant.

The numerical solution of problem (\ref{ibvp}) has been previously tackled in various works. Most of them are based on mixed finite element discretizations, which are known to provide accurate and locally mass conservative velocities, also handling rough permeability coefficients. Examples include multiscale mortar mixed methods, which were analyzed in \cite{kim:par:tho:whe:07} and further studied in \cite{gan:jun:pen:whe:yot:14} from a computational viewpoint in the case of a diagonal tensor $K$. The full tensor case was treated in \cite{gan:pen:whe:wil:yot:12} using expanded mixed methods. A natural extension of this problem considers an additional nonlinearity for $\mathbf{u}$ in (\ref{ibvp:a}) and is usually referred to as Darcy--Forchheimer model. Theoretical and numerical studies dealing with mixed finite element methods for this generalized model can be found in \cite{dou:pae:gio:93,par:05,rui:pan:17}.

A closely related model to problem (\ref{ibvp}) is given by the so-called Richards' equation \cite{ric:31}, which describes the flow of water in variably saturated soils. Richards' equation results from the combination of Darcy's law and the mass balance equation for water, and involves the volumetric water content and the pressure head as unknown variables. If water is assumed to be incompressible, the pressure formulation of Richards' equation has a similar structure to that of equations (\ref{ibvp:a}) and (\ref{ibvp:b}) combined. However, unlike such equations, it may degenerate in certain regions depending on the saturation of the medium. This fact makes the design and analysis of numerical schemes for solving Richards' equation very challenging. Mixed finite element discretizations for this model have been extensively studied in \cite{arb:whe:zha:96,woo:daw:00,rad:pop:kna:04,rad:pop:kna:08,rad:wan:14}. In most of these works, the use of the so-called Kirchhoff transformation permits to eliminate the nonlinearity present in the difussion term, thus simplifying the convergence analysis of the discretization.

In this work, we discuss multipoint flux mixed finite element (MFMFE) methods for problem (\ref{ibvp}). The MFMFE schemes were designed in close relation to the so-called multipoint flux approximation (MPFA) methods. These latter methods, originally proposed as finite volume discretizations \cite{aav:bar:boe:man:98,aav:02,edw:rog:98,edw:02}, consider sub-face (sub-edge in 2D) fluxes that allow for local flux elimination and reduction to a cell-centered pressure scheme. In the MFMFE framework, a similar elimination is performed by using appropriate finite element spaces and suitable quadrature rules. In particular, the lowest order Brezzi--Douglas--Marini ($\mathcal{BDM}_1$) spaces \cite{bre:dou:mar:85} are used on triangles and convex quadrilaterals \cite{whe:yot:06,kla:rad:eig:08}, while the Brezzi--Douglas--Dur\'an--Fortin ($\mathcal{BDDF}_1$) spaces \cite{bre:dou:dur:for:87} are considered on tetrahedra and hexahedra (in their original \cite{whe:yot:06} and an enhanced version \cite{ing:whe:yot:10}, respectively). Alternative formulations on triangular or quadrilateral grids based on a broken Raviart--Thomas space can be found in \cite{kla:rad:eig:08,bau:hof:kna:10,kla:win:06a,kla:win:06,aav:eig:kla:whe:yot:07}. The subsequent definition of certain quadrature rules permits to formulate the original MPFA methods as mixed finite element schemes. Based on the symmetry of the quadrature formula under consideration, we have a symmetric MFMFE method, valid on simplices and $\mathcal{O}(h^2)$-perturbations of parallelograms \cite{kla:rad:eig:08,kla:win:06a,aav:eig:kla:whe:yot:07} and parallelepipeds \cite{ing:whe:yot:10}, and a non-symmetric MFMFE scheme, designed for general quadrilaterals and hexahedra \cite{kla:win:06,whe:xue:yot:12}.

The MFMFE methods, in both their symmetric and non-symmetric variants, have been extensively used and analyzed for incompressible Darcy flow problems; see the preceding references for the steady-state model and \cite{arr:por:yot:14} for the evolutionary model. In addition, they have been applied to Richards' equation \cite{kla:rad:eig:08} and the poroelasticity problem \cite{whe:xue:yot:14}. Remarkably, some experiments that illustrate the performance of MFMFE methods for the slightly compressible flow problem (\ref{ibvp}) can be found in \cite{whe:xue:yot:12a,gan:whe:yot:15}. Nevertheless, in this latter case, a detailed formulation of the methods and their corresponding convergence analysis is lacking so far.

The aim of this work is to extend the existing formulations of MFMFE methods for incompressible Darcy problems to the slightly compressible case. In doing so, we derive a priori error estimates for the continuous-in-time semidiscrete formulation of the symmetric MFMFE scheme. More specifically, the velocity and pressure variables are shown to be first-order convergent on meshes composed of simplices or $\mathcal{O}(h^2)$-perturbations of parallelograms and parallelepipeds. Remarkably, the analysis considers a direct discretization of the original problem, thus avoiding the use of Kirchhoff-type transformations (which are usual for degenerate problems like Richards' equation, as mentioned above). The subsequent application of a suitable time integrator to the resulting differential problem gives rise to a nonlinear system of algebraic equations involving both velocity and pressure unknowns. Following \cite{gan:jun:pen:whe:yot:14}, the solution of this system is based on an inexact Newton method, derived from the fact that the fluid compressibility constant $c_f$ is small for slightly compressible models like (\ref{ibvp}). Further elimination of the velocity unknowns permits to reduce each Newton-type iteration to the solution of a cell-centered linear system for the pressure unknowns. The proposed methods are tested on a variety of numerical examples, including non-constant permeability tensors, smooth and distorted meshes (that require the use of the non-symmetric variant), and random heterogeneous porous media.

The rest of the paper is organized as follows. In Section \ref{section:mfmfe}, we define the MFMFE method for problem (\ref{ibvp}) on different types of spatial meshes in two and three dimensions. The convergence analysis for the continuous-in-time semidiscrete scheme is developed in Section \ref{section:analysis}. The next section describes the fully discrete scheme and its subsequent reduction to a nonlinear system of residual equations. In Section \ref{section:residual:eqs}, we present an inexact Newton method that yields an efficient solution of the preceding system in terms of the pressure unknowns. Finally, Section \ref{section:numer:examples} illustrates the theoretical results with numerical experiments, and presents and application of the proposed methods on a quarter five-spot configuration problem.

\section{The multipoint flux mixed finite element method}\label{section:mfmfe}

In this section, we define the MFMFE method for (\ref{ibvp}) on simplicial, quadrilateral and hexahedral elements. This method is related to MPFA-type schemes through the use of adequate mixed finite element spaces and a special quadrature rule. Further details are provided in the sequel.

\subsection{The weak formulation}

For a domain $G\subset\mathbb{R}^d$, let $W^{k,p}(G)$ be the standard Sobolev space, with $k\in\mathbb{R}$ and $1\leq p\leq\infty$, endowed with the norm $\|\cdot\|_{k,p,G}$. Let $H^k(G)$ be the Hilbert space $W^{k,2}(G)$, equipped with the norm $\|\cdot\|_{k,G}$. We further denote by $(\cdot,\cdot)_G$ and $\|\cdot\|_G$ the inner product and norm, respectively, in either $L^2(G)$ or $(L^2(G))^d$. The subscript $G$ will be omitted whenever $G=\Omega$. For a section of the domain or element boundary $S\subset\mathbb{R}$, $\langle\cdot,\cdot\rangle_S$ and $\|\cdot,\cdot\|_S$ represent the $L^2(S)$-inner product (or duality pairing) and norm, respectively. We shall also use the space
$$
H(\mbox{div};G)=\{\mathbf{v}\in(L^2(G))^d:\nabla\cdot\mathbf{v}\in L^2(G)\},
$$
with corresponding norm $\|\mathbf{v}\|_{\mathrm{div};\,G}=(\|\mathbf{v}\|^2_G+\|\nabla\cdot\mathbf{v}\|^2_G)^{1/2}$. Finally, if $\chi(G)$ denotes any of the above normed spaces on $G$, with associated norm $\|\cdot\|_{\chi(G)}$, we shall consider $L^q(J;\chi(G)):=L^q([0,T];\chi(G))$ as the space of $\chi$-valued functions $\varphi:[0,T]\rightarrow\chi(G)$, endowed with the norm
$$
\|\varphi\|_{L^q(J;\chi(G))}:=\begin{cases}
\,\left(\displaystyle\int_0^T{\|\varphi(t)\|_{\chi(G)}^q\,dt}\right)^{1/q}&
\mathrm{if}\,\,1\leq q < \infty,\\[2.5ex]
\,\mathrm{ess}\,\sup_{t\in[0,T]}\|\varphi(t)\|_{\chi(G)}
&\mathrm{if}\,\, q = \infty.
\end{cases}
$$

The following assumptions are made on (\ref{ibvp}) for the subsequent convergence analysis: there exist positive constants $\alpha_1$, $\alpha_2$, $\kappa_{\ast}$, $\kappa^{\ast}$, $\gamma_1$, $\gamma_2$ and $\eta$, defined to be independent of the discretization parameters, such that
\begin{enumerate}
	\item[(A1)] $\alpha_1\leq\phi(\textbf{x})\leq\alpha_2$, for any $\mathbf{x}\in\Omega$;\\[-1.5ex]
	\item[(A2)] $\kappa_{\ast}\,\mathbf{\xi}^T\mathbf{\xi}\leq
	\mathbf{\xi}^{T}K(\mathbf{x})\,\mathbf{\xi}\leq
	\kappa^{\ast}\,\mathbf{\xi}^T\mathbf{\xi}$, for any $\mathbf{x}\in\Omega$ and $\mathbf{\xi}\neq\mathbf{0}\in\mathbb{R}^d$;\\[-1.5ex]
	\item[(A3)] $\rho$, $\rho'$ exist, are continuous and $\gamma_1\leq\rho(\cdot),\rho'(\cdot)\leq\gamma_2$; in addition, $\rho''$ exists and $|\rho''(\cdot)|\leq\eta$.
\end{enumerate}
Regarding the validity of these assumptions, note that a bounded and positive porosity, and a bounded and uniformly positive definite permeability tensor are physically reasonable (see, e.g., \cite{bea:88}). Furthermore, the assumption (A3) on the fluid density is applicable to liquids, unless they contain large quantities of dissolved gas \cite{pea:77}. Similar hypotheses are considered in \cite[Section 3.4]{whe:73}.

In this context, the variational formulation of the first-order system (\ref{ibvp}) reads: \emph{Find} $(\mathbf{u},p):[0,T]\rightarrow V\times W$ \emph{such that}
\begin{subequations}\label{weak:mixed:formulation}
	\renewcommand{\theequation}{\theparentequation\alph{equation}}
	\begin{align}
	&(K^{-1}\rho^{-1}(p)\,\mathbf{u},\mathbf{v})=(p,\nabla\cdot\mathbf{v})+(\rho(p)\,\mathbf{g},\mathbf{v}),&&\mathbf{v}\in V,\label{weak:mixed:formulation:a}\\[0.5ex]
	&(\phi\,\rho(p)_t,w)+(\nabla\cdot\mathbf{u},w)=(f,w),&&w\in W,\label{weak:mixed:formulation:b}\\[0.5ex]
	&p(0)=p_{0},&&\label{weak:mixed:formulation:c}\
	\end{align}
\end{subequations}
where $V=\left\{\mathbf{v}\in {H}(\mbox{div};\Omega):\mathbf{v}\cdot\mathbf{n}=0\ \mathrm{on}\ \Gamma_N\right\}$ and $W=L^2({\Omega})$. Throughout this paper, we will denote by $C$ any generic positive constant defined to be independent of the discretization parameters.

\subsection{Mixed finite element spaces}

Let $\mathcal{T}_h$ be a conforming, shape-regular and quasi-uniform partition of $\Omega$ into triangles or convex quadrilaterals, in two dimensions, and tetrahedra or hexahedra, in three dimensions. In all the cases, we denote $h=\max_{E\in\mathcal{T}_h}\mathrm{diam}(E)$. For any element $E\in\mathcal{T}_h$, we introduce a bijective mapping $F_E$ such that $F_E(\hat{E})=E$. We further define, for each mapping $F_E$, the Jacobian matrix $DF_E$ and its determinant $J_E=|\det(DF_E)|$. The inverse mapping is denoted by $F_E^{-1}$ and its Jacobian matrix is given by $DF_E^{-1}(\mathbf{x})=(DF_E(\hat{\mathbf{x}}))^{-1}$ with determinant $J^{-1}_E(\mathbf{x})=1/J_E(\hat{\mathbf{x}})$. For the sake of completeness, we briefly recall in the sequel how to construct the corresponding mixed finite element spaces for the elements under consideration.

\textbf{Simplicial elements.} Let us consider the reference tetrahedron $\hat{E}$ with vertices $\hat{\mathbf{r}}_1=(0,0,0)^T$, $\hat{\mathbf{r}}_2=(1,0,0)^T$, $\hat{\mathbf{r}}_3=(0,1,0)^T$ and $\hat{\mathbf{r}}_4=(0,0,1)^T$. Let $\mathbf{r}_i=(x_i,y_i,z_i)^T$ be the corresponding vertices of $E$, for $i=1,2,3,4$. We define $F_E$ as an affine mapping of the form
\begin{equation*}
F_E(\hat{\mathbf{x}})=\mathbf{r}_1(1-\hat{x}-\hat{y}-\hat{z})+\mathbf{r}_2\hat{x}+\mathbf{r}_3\hat{y}+\mathbf{r}_4\hat{z}.
\end{equation*}
In the case of triangles, we consider a two-dimensional counterpart of these expressions. Let $\hat{V}(\hat{E})$ and $\hat{W}(\hat{E})$ be the finite element spaces on the reference element $\hat{E}$. In particular, we use the $\mathcal{BDM}_1$ spaces on the unit triangle and the $\mathcal{BDDF}_1$ spaces on the unit tetrahedron, i.e.,
\begin{equation*}
\hat{V}(\hat{E})=(\mathbb{P}_1(\hat{E}))^d,\qquad
\hat{W}(\hat{E})=\mathbb{P}_0(\hat{E}),
\end{equation*}
where $\mathbb{P}_k$ denotes the space of polynomials of degree not greater than $k$.

\textbf{Convex quadrilateral elements.} In this case, $\hat{E}$ is the unit square with vertices $\hat{\mathbf{r}}_1=(0,0)^T$, $\hat{\mathbf{r}}_2=(1,0)^T$, $\hat{\mathbf{r}}_3=(1,1)^T$ and $\hat{\mathbf{r}}_4=(0,1)^T$. The corresponding vertices of $E$ are denoted by $\mathbf{r}_i=(x_i,y_i)^T$, for $i=1,2,3,4$. Let $F_E$ be a bilinear mapping given by
$$
F_E(\hat{\mathbf{x}})=\mathbf{r}_1(1-\hat{x})(1-\hat{y})+\mathbf{r}_2\hat{x}(1-\hat{y})+\mathbf{r}_3\hat{x}\hat{y}+\mathbf{r}_4(1-\hat{x})\hat{y}.
$$
On the unit square, we define $\mathcal{BDM}_1$ spaces of the form
\begin{equation*}
\hat{V}(\hat{E})=(\mathbb{P}_1(\hat{E}))^2+r\,\mathrm{curl}(\hat{x}^2\hat{y})+s\,\mathrm{curl}(\hat{x}\hat{y}^2),\qquad
\hat{W}(\hat{E})=\mathbb{P}_0(\hat{E}),\nonumber
\end{equation*}
where $r$ and $s$ are real constants.

\textbf{Hexahedral elements.} In the case of hexahedra, $\hat{E}$ denotes the unit cube with vertices $\hat{\mathbf{r}}_1=(0,0,0)^T$, $\hat{\mathbf{r}}_2=(1,0,0)^T$, $\hat{\mathbf{r}}_3=(1,1,0)^T$, $\hat{\mathbf{r}}_4=(0,1,0)^T$, $\hat{\mathbf{r}}_5=(0,0,1)^T$, $\hat{\mathbf{r}}_6=(1,0,1)^T$, $\hat{\mathbf{r}}_7=(1,1,1)^T$ and $\hat{\mathbf{r}}_8=(0,1,1)^T$. Let $\mathbf{r}_i$ be the corresponding vertices of $E$, for $i=1,2,\ldots,8$. Note that the hexahedra may have non-planar faces. In this case, $F_E$ is a trilinear mapping of the form
\begin{figure}[t!]
	\begin{center}
		\includegraphics[scale=0.25]{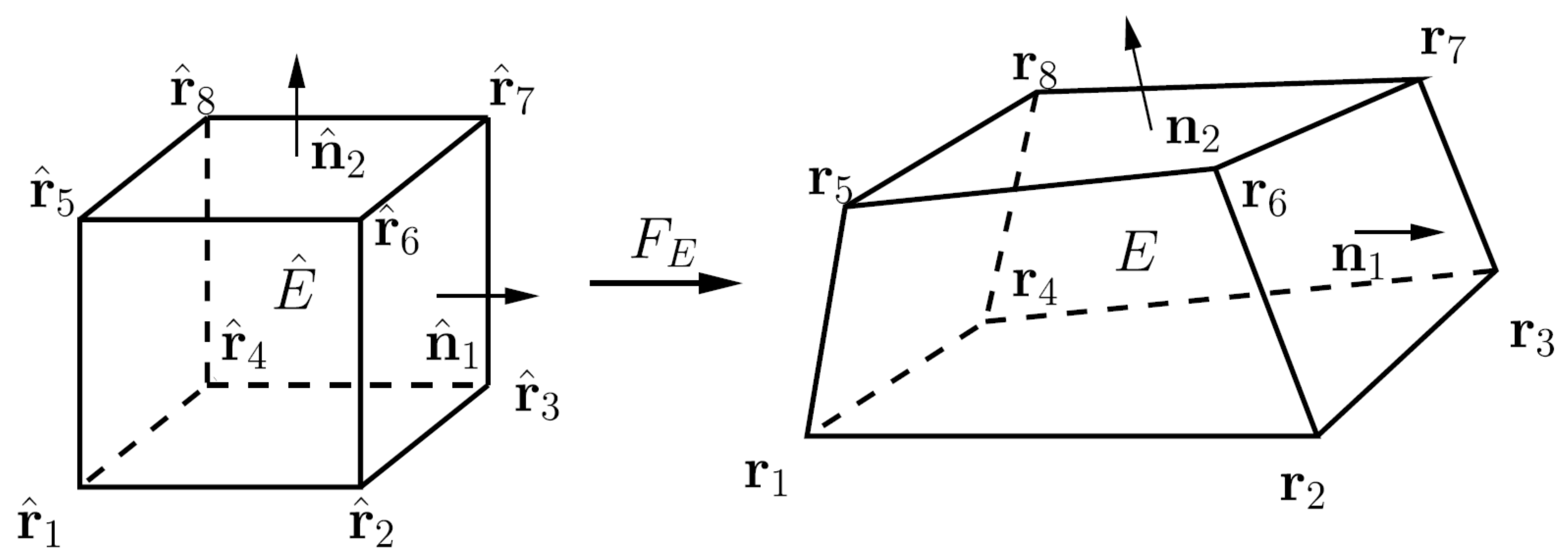}
	\end{center}\caption{Trilinear hexahedral mapping.}\label{fig:mapping}
\end{figure}
\begin{align}\label{trilinear:mapping}
&F_E(\hat{\mathbf{x}})=\mathbf{r}_1(1-\hat{x})(1-\hat{y})(1-\hat{z})+\mathbf{r}_2\hat{x}(1-\hat{y})(1-\hat{z})+\mathbf{r}_3\hat{x}\hat{y}(1-\hat{z})\nonumber\\
&\hspace*{1cm}+\mathbf{r}_4(1-\hat{x})\hat{y}(1-\hat{z})+\mathbf{r}_5(1-\hat{x})(1-\hat{y})\hat{z}+\mathbf{r}_6\hat{x}(1-\hat{y})\hat{z}+\mathbf{r}_7\hat{x}\hat{y}\hat{z}+\mathbf{r}_8(1-\hat{x})\hat{y}\hat{z}.
\end{align}
Figure \ref{fig:mapping} shows the trilinear mapping from the reference element $\hat{E}$ to a physical element $E$. In the figure, the outward unit normal vectors to the faces of $\hat{E}$ and $E$ are denoted by $\hat{\mathbf{n}}_i$ and $\mathbf{n}_i$, respectively, for $i=1,2,\ldots,8$. We shall also use the notation $\hat{\mathbf{n}}_{\hat{e}}$ and $\mathbf{n}_e$ to represent the outward unit normals on faces $\hat{e}\subset\partial\hat{E}$ and $e\subset\partial E$, respectively. On the unit cube, we use the so-called enhanced $\mathcal{BDDF}_1$ spaces introduced in \cite{ing:whe:yot:10}:
\begin{equation*}
\begin{array}{rll}
&\hat{V}(\hat{E})=\mathcal{BDDF}_1(\hat{E})+r_2\,\mathrm{curl}(0,0,\hat{x}^2\hat{z})^T+r_3\,\mathrm{curl}(0,0,\hat{x}^2\hat{y}\hat{z})^T+s_2\,\mathrm{curl}(\hat{x}\hat{y}^2,0,0)^T\\[0.5ex]
&\hspace*{3.5cm}+s_3\,\mathrm{curl}(\hat{x}\hat{y}^2\hat{z},0,0)^T+t_2\,\mathrm{curl}(0,\hat{y}\hat{z}^2,0)^T+t_3\,\mathrm{curl}(0,\hat{x}\hat{y}\hat{z}^2,0)^T,\\[0.75ex]
&\hat{W}(\hat{E})=\mathbb{P}_0(\hat{E}),
\end{array}
\end{equation*}
where the $\mathcal{BDDF}_1(\hat{E})$ space is given by
\begin{align}
&\mathcal{BDDF}_1(\hat{E})=(\mathbb{P}_1(\hat{E}))^3+r_0\,\mathrm{curl}(0,0,\hat{x}\hat{y}\hat{z})^T+r_1\,\mathrm{curl}(0,0,\hat{x}\hat{y}^2)^T+s_0\,\mathrm{curl}(\hat{x}\hat{y}\hat{z},0,0)^T\nonumber\\
&\hspace*{4.2cm}+s_1\,\mathrm{curl}(\hat{y}\hat{z}^2,0,0)^T+t_0\,\mathrm{curl}(0,\hat{x}\hat{y}\hat{z},0)^T+t_1\,\mathrm{curl}(0,\hat{x}^2\hat{z},0)^T,\nonumber
\end{align}
$r_i,s_i$ and $t_i$ being real constants, for $i=0,1,2,3$.

Note that, in all four cases, $\hat{\nabla}\cdot\hat{V}(\hat{E})=\hat{W}(\hat{E})$. Furthermore, on any face (edge in 2D) $\hat{e}\subset\partial\hat{E}$, $\hat{\mathbf{v}}\in\hat{V}(\hat{E})$ is such that $\hat{\mathbf{v}}\cdot\hat{\mathbf{n}}_{\hat{e}}\in\mathbb{P}_1(\hat{e})$ on the reference simplex or the reference square, and $\hat{\mathbf{v}}\cdot\hat{\mathbf{n}}_{\hat{e}}\in\mathbb{Q}_1(\hat{e})$ on the reference cube. Here, $\mathbb{Q}_1(\hat{e})$ denotes the space of bilinear functions on $\hat{e}$. The degrees of freedom for $\hat{\mathbf{v}}\in\hat{V}(\hat{E})$ are chosen to be the values of $\hat{\mathbf{v}}\cdot\hat{\mathbf{n}}_{\hat{e}}$ at the vertices of each face (edge) $\hat{e}$; see Figure \ref{fig:dofs} in the hexahedral case. 

\begin{figure}[t]
	\begin{center}
		\includegraphics[scale=0.26]{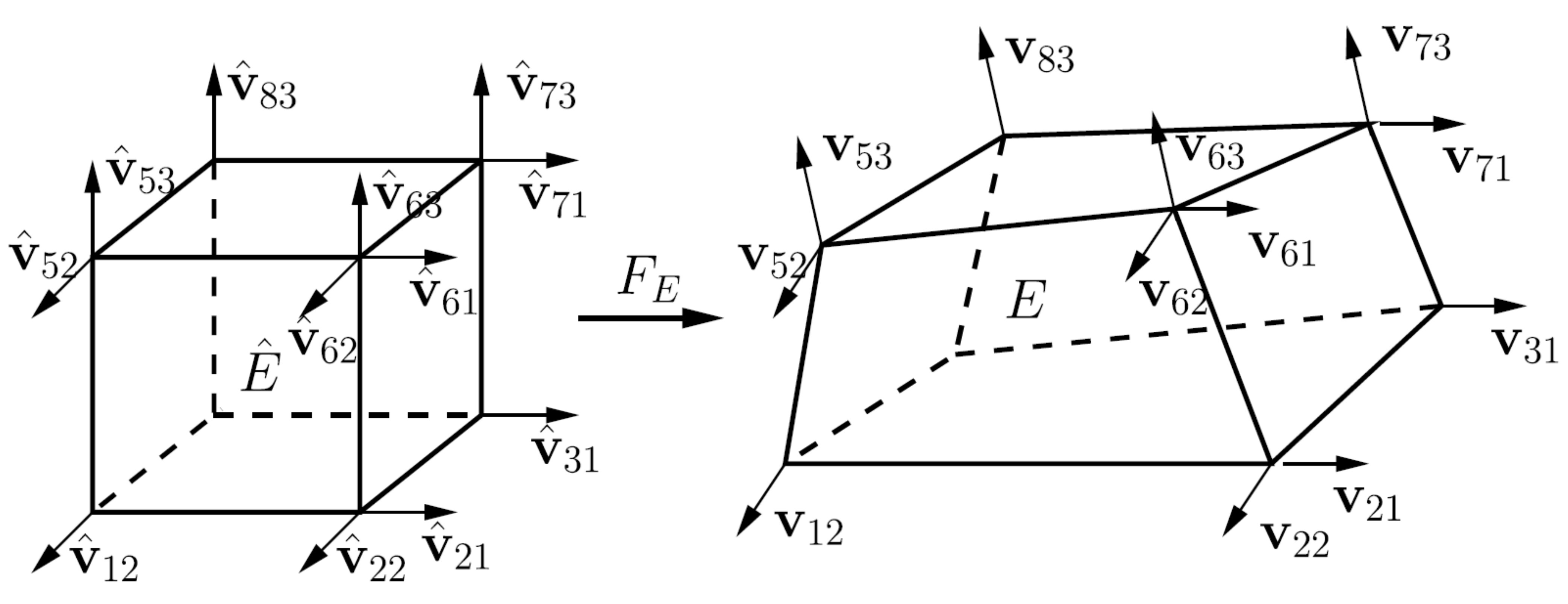}
	\end{center}\caption{Degrees of freedom and basis functions for the enhanced $\mathcal{BDDF}_1$ velocity space on hexahedra.}\label{fig:dofs}
\end{figure}

The spaces $V(E)$ and $W(E)$ on any physical element $E\in\mathcal{T}_h$ are defined via the transformations
\begin{equation}\label{piola}
\mathbf{v}\leftrightarrow\hat{\mathbf{v}}:\mathbf{v}=\left(J_E^{-1}{DF}_E\,\hat{\mathbf{v}}\right)\circ{F}_E^{-1},\qquad
w\leftrightarrow\hat{w}: w = \hat{w}\circ{F}_E^{-1}.
\end{equation}
The former is known as the Piola transformation \cite{tho:77} and it is defined to preserve the continuity of the normal components of velocity vectors across interelement edges. This is a necessary condition that must be fulfilled when building approximations to $H(\mathrm{div};\Omega)$. The Piola transformation satisfies the following properties \cite{bre:for:91}
\begin{equation}\label{piola:properties}
(\nabla\cdot \mathbf{v},w)_{E}=(\hat{\nabla}\cdot\hat{\mathbf{v}},\hat{w})_{\hat{E}},\qquad\langle\mathbf{v}\cdot\mathbf{n}_e,w\rangle_e=
\langle\hat{\mathbf{v}}\cdot\hat{\mathbf{n}}_{\hat{e}},\hat{w}\rangle_{\hat{e}},
\end{equation}
where $(\mathbf{v},w)\in V(E)\times W(E)$ and $(\hat{\mathbf{v}},\hat{w})\in\hat{V}(\hat{E})\times\hat{W}(\hat{E})$. Moreover, it is well known that (cf. \cite{whe:yot:06})
\begin{equation*}
\mathbf{v}\cdot\mathbf{n}_e=\left(\frac{1}{|J_E\,DF^{-T}_E\hat{\mathbf{n}}_{\hat{e}}|_{\mathbb{R}^d}}\,\hat{\mathbf{v}}\cdot\hat{\mathbf{n}}_{\hat{e}}\right)\circ{F}_E^{-1}(\mathbf{x}),\qquad
\nabla\cdot\mathbf{v}=\left(J_E^{-1}\,\hat{\nabla}\cdot\hat{\mathbf{v}}\right)\circ{F}_E^{-1}(\mathbf{x}).
\end{equation*}
For any $\mathbf{v}\in V(E)$, while $J_E$ is constant on simplices, this is not true on quadrilaterals and hexahedra. As a result, $\nabla\cdot\mathbf{v}$ is not constant in the latter cases. Along these lines, $\mathbf{v}\cdot\mathbf{n}_{e}\in\mathbb{P}_1(e)$ on simplices and quadrilaterals, but $\mathbf{v}\cdot\mathbf{n}_{e}\not\in\mathbb{Q}_1(e)$ on hexahedra. Finally, the global MFE spaces $V_h\times W_h\subset V\times W$ on $\mathcal{T}_h$ are given by
\begin{align}
&V_h=\left\{\mathbf{v}\in V: \mathbf{v}|_E\leftrightarrow\hat{\mathbf{v}},\,\hat{\mathbf{v}}\in \hat{V}(\hat{E})\,\,\forall\,E\in\mathcal{T}_h\right\},\nonumber\\[0.5ex]
&W_h=\{w\in W: w|_E\leftrightarrow\hat{w},\,\hat{w}\in\hat{W}(\hat{E})\,\,\forall\,E\in\mathcal{T}_h\}.\nonumber
\end{align}

To conclude, let us now recall how to construct the projection operator onto the space $V_h$. We first introduce a reference element projection operator $\hat{\Pi}:(H^1(\hat{E}))^d\rightarrow\hat{V}(\hat{E})$, which is defined, for any $\hat{\mathbf{q}}\in(H^1(\hat{E}))^d$, as
\begin{equation*}
\langle(\hat{\Pi}\hat{\mathbf{q}}-\hat{\mathbf{q}})\cdot\hat{\mathbf{n}}_{\hat{e}},\hat{q}_1\rangle=0,
\end{equation*}
on any edge $\hat{e}\subset\partial\hat{E}$, where $\hat{q}_1\in\mathbb{P}_1(\hat{e})$. Then, the global projection operator $\Pi_h:(H^1(\Omega))^d\,\cap\,V\rightarrow V_h$ is locally defined on each element $E$ via the Piola transformation (\ref{piola}), i.e., for any $\mathbf{q}\in(H^1(\Omega))^d\cap V$, $\Pi_h\mathbf{q}|_E\leftrightarrow\widehat{\Pi_h\mathbf{q}}=\hat{\Pi}\hat{\mathbf{q}}\in\hat{V}(\hat{E})$. Based on the previous expressions, it can be deduced that
\begin{equation}\label{div:projection}
(\nabla\cdot(\Pi_h\mathbf{q}-\mathbf{q}),w)=0\qquad\forall\,w\in W_h.
\end{equation}
In addition, we also require the standard $L^2(\Omega)$-projection operator onto the space $W_h$. To define it, we first introduce the $L^2(\hat{E})$-projection $\hat{\mathcal{P}}:L^2(\hat{E})\rightarrow\hat{W}(\hat{E})$ satisfying, for any $\hat{\varphi}\in\hat{W}(\hat{E})$,
\begin{equation*}
(\hat{\varphi}-\hat{\mathcal{P}}\hat{\varphi},\hat{w})_{\hat{E}}=0\qquad\forall\,\hat{w}\in\hat{W}(\hat{E}).
\end{equation*}
Then, we let $\mathcal{P}_h:L^2(\Omega)\rightarrow W_h$ be the $L^2(\Omega)$-projection operator, which is locally defined on each element $E$, for any $\varphi\in L^2(\Omega)$, as $\mathcal{P}_h\varphi|_E=\hat{\mathcal{P}}\hat{\varphi}\circ F_E^{-1}$. It is not difficult to see that, due to (\ref{piola:properties}),
\begin{equation}\label{l2:projection:property}
(\varphi-\mathcal{P}_h\varphi,\nabla\cdot\mathbf{v})=0\qquad\forall\,\mathbf{v}\in V_h.
\end{equation}

\subsection{A quadrature rule}

In the mixed variational formulation (\ref{weak:mixed:formulation}), it is necessary to compute an integral of the form $(K^{-1}\rho^{-1}(p)\,\mathbf{q},\mathbf{v})$, for $\mathbf{q},\mathbf{v}\in V_h$. In doing so, the MFMFE method considers a quadrature rule that allows for local velocity elimination. This idea has been previously used by several authors in the context of elliptic problems.

In the case of the Laplace operator on triangular and regular tetrahedral grids, the first successful attempt to eliminate the velocity unknowns in order to derive a system for the pressures was proposed in \cite{bar:mai:oud:96}. In this work, a cell-centered finite volume scheme with a 4-cell stencil was derived from the lowest order Raviart--Thomas mixed formulation by using a suitable quadrature rule (defined to be exact for constant functions). Subsequently, this technique was extended to reaction--diffusion problems involving a diagonal tensor coefficient in \cite{mic:sac:sal:01}. In this case, the use of different quadrature formulas led to a family of cell-centered finite volume schemes with a 4-cell stencil. More recently, an overview of elimination methods for the Poisson equation on triangular and tetrahedral elements has been provided in \cite{bre:for:mar:06}. In this work, the authors further show the close relationship between the derived schemes and some classical finite volume formulations.

Regarding the use of rectangular grids, these strategies were first addressed in \cite{rus:whe:83}. In the diagonal tensor case, the authors showed that the lowest order Raviart--Thomas mixed method could be reduced to a cell-centered finite difference scheme with a 5-cell stencil by using a combination of the midpoint and trapezoidal rules. Later on, the newly derived scheme was proven to preserve the convergence properties of the original mixed finite element method in \cite{wei:whe:88}. These ideas were extended to the (possibly discontinuous) full tensor case in \cite{arb:whe:yot:97}. In doing so, the expanded mixed finite element formulation was considered in combination with similar quadrature rules to obtain a cell-centered scheme with a 9-cell stencil on rectangles and a 19-cell stencil on rectangular parallelepipeds. The generalization to smooth logically rectangular and cubic grids was considered in \cite{arb:daw:kee:whe:yot:98}.

In our setting, the integration on each element $E\in\mathcal{T}_h$ is performed by mapping to the reference element $\hat{E}$, where this quadrature rule is defined, i.e.,
\begin{equation*}
(K^{-1}\rho^{-1}(p)\,\mathbf{q},\mathbf{v})_{E}=\left(J_E^{-1}DF_E^TK^{-1}\rho^{-1}(p)\,DF_E\,\hat{\mathbf{q}},\hat{\mathbf{v}}\right)_{\hat{E}}
=:(\mathcal{K}_E^{-1}\hat{\mathbf{q}},\hat{\mathbf{v}})_{\hat{E}},
\end{equation*}
for any $\mathbf{q}$, $\mathbf{v}\in V_h$ and $\hat{\mathbf{q}}$, $\hat{\mathbf{v}}\in\hat{V}(\hat{E})$, where $\mathcal{K}_E^{-1}=J_E^{-1}DF_E^TK^{-1}\rho^{-1}(p)\,DF_E$. The quadrature rule on $E\in\mathcal{T}_h$ is given by the trapezoidal rule, i.e.,
\begin{equation*}
(K^{-1}\rho^{-1}(p)\,\mathbf{q},\mathbf{v})_{Q,E}:=(\mathcal{K}_E^{-1}\hat{\mathbf{q}},\hat{\mathbf{v}})_{\hat{Q},\hat{E}}:=
\frac{|\hat{E}|}{n_v}\sum_{i=1}^{n_v}\mathcal{K}_E^{-1}(\hat{\mathbf{r}}_i)\hat{\mathbf{q}}(\hat{\mathbf{r}}_i)\cdot\hat{\mathbf{v}}(\hat{\mathbf{r}}_i),
\end{equation*}
where $|\hat{E}|$ is the volume (area) of $\hat{E}$ and $n_v$ denotes the number of vertices of $\hat{E}$ (i.e., $n_v=3$ for the unit triangle, $n_v=4$ for the unit square or the unit tetrahedron and $n_v=8$ for the unit cube). Hence, the global quadrature rule is defined as
\begin{equation}\label{global:quadrature}
(K^{-1}\rho^{-1}(p)\,\mathbf{q},\mathbf{v})_{Q}:=\sum_{E\in\mathcal{T}_h}(K^{-1}\rho^{-1}(p)\,\mathbf{q},\mathbf{v})_{Q,E}.
\end{equation}
Mapping back to the physical element $E$, we obtain
\begin{equation}\label{symmetric:quadrature:physical}
(K^{-1}\rho^{-1}(p)\,\mathbf{q},\mathbf{v})_{Q,E}=\frac{1}{n_v}\sum_{i=1}^{n_v}J_E(\hat{\mathbf{r}}_i){K}^{-1}(\mathbf{r}_i)\rho^{-1}(p(\mathbf{r}_i))\,\mathbf{q}(\mathbf{r}_i)\cdot\mathbf{v}(\mathbf{r}_i).
\end{equation}
Note that the expressions (\ref{global:quadrature})-(\ref{symmetric:quadrature:physical}) induce a symmetric discrete bilinear form in $V_h$. Following \cite[Lemma 2.4]{whe:yot:06} and \cite[Lemma 2.6]{ing:whe:yot:10}, together with the hypotheses (A2)-(A3), it is easy to check that such a bilinear form is coercive in $V_h$, i.e.,
\begin{equation}\label{coercivity:quadrature}
(K^{-1}\rho^{-1}(p)\,\mathbf{q},\mathbf{q})_Q\geq C\|\mathbf{q}\|^2\qquad\forall\,\mathbf{q}\in V_h.
\end{equation}
As a result, $(K^{-1}\rho^{-1}(p)\,\mathbf{q},\mathbf{v})_Q$ is an inner product and $(K^{-1}\rho^{-1}(p)\,\mathbf{q},\mathbf{q})_Q^{1/2}$ is a norm in $V_h$. Hence, the continuity in $V_h$ is also satisfied, i.e.,
\begin{equation}\label{continuity:quadrature}
|(K^{-1}\rho^{-1}(p)\,\mathbf{q},\mathbf{v})_Q|\leq C\|\mathbf{q}\|\|\mathbf{v}\|\qquad\forall\,\mathbf{q},\mathbf{v}\in V_h.
\end{equation}

\subsection{The semidiscrete scheme}

The MFMFE approximation to (\ref{weak:mixed:formulation}) is given by: \emph{Find} $(\mathbf{u}_h,p_h):[0,T]\rightarrow {V_h}\times {W_h}$ \emph{such that}
\begin{subequations}\label{mfmfe:method}
	\renewcommand{\theequation}{\theparentequation\alph{equation}}
	\begin{align}
	&(K^{-1}\rho^{-1}(p_h)\,\mathbf{u}_h,\mathbf{v})_Q=(p_h,\nabla\cdot\mathbf{v})
	+(\rho(p_h)\,\mathbf{g},\mathbf{v}),&&\mathbf{v}\in V_h,\label{mfmfe:method:a}\\[0.5ex]
	&(\phi\,\rho(p_h)_t,w)+(\nabla\cdot\mathbf{u}_h,w)=(f,w),&&w\in W_h,\label{mfmfe:method:b}\\[0.5ex]
	&p_h(0)=\tilde{p}_h(0),&&\label{mfmfe:method:c}\
	\end{align}
\end{subequations}
where $\tilde{p}_h(0)$ denotes the elliptic MFE projection of $p_0$ (to be defined below). Note that the initial condition $p_h(0)$ determines $\mathbf{u}_h(0)$ through (\ref{mfmfe:method:a}). As mentioned above, in the context of incompressible flow problems, this method is sometimes referred to as the symmetric MFMFE scheme, since the discrete bilinear form given by (\ref{global:quadrature})-(\ref{symmetric:quadrature:physical}) is symmetric (see, e.g., \cite{whe:xue:yot:12,arr:por:yot:14}).

\section{Convergence analysis of the semidiscrete scheme}\label{section:analysis}

In this section, we obtain a priori error estimates for the velocity and pressure variables of the MFMFE semidiscrete formulation (\ref{mfmfe:method}).

\subsection{The elliptic mixed finite element projection}

Let us first define an elliptic projection operator into the mixed finite element spaces. Following \cite{kim:par:tho:whe:07,whe:73}, such a projection is given by the map: \emph{Find} $(\tilde{\mathbf{u}}_h,\tilde{p}_h):[0,T]\rightarrow {V_h}\times{W_h}$ \emph{such that}
\begin{subequations}\label{elliptic:projection}
	\renewcommand{\theequation}{\theparentequation\alph{equation}}
	\begin{align}
	&(K^{-1}\rho^{-1}(p)\,\tilde{\mathbf{u}}_h,\mathbf{v})_Q=(\tilde{p}_h,\nabla\cdot\mathbf{v})+(\rho(p)\,\mathbf{g},\mathbf{v}),&&\mathbf{v}\in V_h,\label{elliptic:projection:a}\\[0.5ex]
	&(\nabla\cdot\tilde{\mathbf{u}}_h,w)=(f-\phi\,\rho(p)_t,w),&&w\in W_h,\label{elliptic:projection:b}\\[0.5ex]
	&(\tilde{p}_h(0),w)=(p_0,w),&&w\in W_h,\label{elliptic:projection:c}\
	\end{align}
\end{subequations}
Note that the pair $(\tilde{\mathbf{u}}_h,\tilde{p}_h)$ is precisely the solution of the mixed finite element approximation to a continuous elliptic problem whose exact solution is $(\mathbf{u},p)$. Subtracting (\ref{elliptic:projection}) from (\ref{weak:mixed:formulation}), we get the error equations
\begin{align*}
&(K^{-1}\rho^{-1}(p)\,\mathbf{u},\mathbf{v})
-(K^{-1}\rho^{-1}(p)\,\tilde{\mathbf{u}}_h,\mathbf{v})_Q
=(p-\tilde{p}_h,\nabla\cdot\mathbf{v}),&&\mathbf{v}\in{V_h},\\[0.5ex]
&(\nabla\cdot(\mathbf{u}-\tilde{\mathbf{u}}_h),w)=0,&&w\in W_h,\\[0.5ex]
&(p(0)-\tilde{p}_h(0),w)=0,&&w\in W_h.
\end{align*}
The local quadrature error on each element is defined to be
$$
\sigma_E(\mathbf{q},\mathbf{v}):=(\mathbf{q},\mathbf{v})_E-(\mathbf{q},\mathbf{v})_{Q,E}
$$
in such a way that $\sigma(\mathbf{q},\mathbf{v})|_E:=\sigma_E(\mathbf{q},\mathbf{v})$ represents the global quadrature error. Taking into account (\ref{div:projection}) and (\ref{l2:projection:property}), the previous equations can be rewritten as
\begin{subequations}\label{error:elliptic}
	\renewcommand{\theequation}{\theparentequation\alph{equation}}
	\begin{align}
	&(K^{-1}\rho^{-1}(p)(\Pi_h\mathbf{u}-\tilde{\mathbf{u}}_h),\mathbf{v})_Q
	=(\mathcal{P}_hp-\tilde{p}_h,\nabla\cdot\mathbf{v})\nonumber\\
	&\hspace*{2cm}-(K^{-1}\rho^{-1}(p)\,\mathbf{u},\mathbf{v})
	+(K^{-1}\rho^{-1}(p)\,\Pi_h\mathbf{u},\mathbf{v})_Q,
	&&\mathbf{v}\in{V_h},\label{error:elliptic:a}\\[0.5ex]
	&(\nabla\cdot(\Pi_h\mathbf{u}-\tilde{\mathbf{u}}_h),w)=0.&&w\in W_h,\label{error:elliptic:b}
	\end{align}
\end{subequations}

In the sequel, we recall the error estimates for the elliptic projection (\ref{elliptic:projection}). The results assume certain restrictions on the element geometry (in the case of quadrilaterals and hexahedra), which are described in the sequel. Following the terminology from \cite{whe:yot:06,ing:whe:yot:10,whe:xue:yot:12b}, we call generalized quadrilaterals the (possibly non-planar) faces of a hexahedral element $E$ defined via a trilinear mapping $F_E$ of the form (\ref{trilinear:mapping}). A generalized quadrilateral with vertices $\mathbf{r}_1$, $\mathbf{r}_2$, $\mathbf{r}_3$ and $\mathbf{r}_4$ is called an $h^2$-parallelogram if \cite{ewi:liu:wan:99}
\begin{equation*}
|\mathbf{r}_{1}-\mathbf{r}_{2}+\mathbf{r}_{3}-\mathbf{r}_{4}|_{\mathbb{R}^d}\leq Ch^2.
\end{equation*}
Elements of this kind are obtained by uniform refinements of a general quadrilateral grid. Furthermore, a hexahedral element is called an $h^2$-parallelepiped if all of its faces are $h^2$-parallelograms. Based on (\ref{trilinear:mapping}), this condition implies that $\partial_{\hat{x}\hat{y}}F_E$, $\partial_{\hat{x}\hat{z}}F_E$ and $\partial_{\hat{y}\hat{z}}F_E$ are $\mathcal{O}(h^2)$.

In the following lemma, $W^{\alpha,\infty}_{\mathcal{T}_h}$ denotes a space consisting of functions $\varphi$ such that $\varphi|_E\in W^{\alpha,\infty}(E)$ for all $E\in\mathcal{T}_h$, $\alpha$ being an integer.

\begin{lemma}\label{lemma:velocity:pressure:elliptic:projection}
	If $K^{-1}$ and $\rho^{-1}\in W^{1,\infty}_{\mathcal{T}_h}$, then the velocity $\tilde{\mathbf{u}}_h$ and the pressure $\tilde{p}_h$ of the MFMFE elliptic projection (\ref{elliptic:projection}), with the quadrature rule (\ref{global:quadrature})-(\ref{symmetric:quadrature:physical}), satisfy, for all $t\in[0,T]$,
	\begin{align}
	&\|(\mathbf{u}-\tilde{\mathbf{u}}_h)(t)\|\leq Ch\,\|\mathbf{u}\|_1,\label{u:elliptic}\\[0.5ex]
	&\|(p-\tilde{p}_h)(t)\|\leq Ch\,(\|\mathbf{u}\|_1+\|p\|_1)\label{p:elliptic}
	\end{align}
	on simplices, $h^2$-parallelograms and $h^2$-parallelepipeds, where $C$ is a positive constant, defined to be independent of $h$.
\end{lemma}
\begin{proof}
	The error equations in the form (\ref{error:elliptic}) are the counterpart of those considered in \cite{whe:yot:06,ing:whe:yot:10} for the convergence analysis of the incompressible case. Hence, the result follows from \cite[Theorems 3.4 and 4.1]{whe:yot:06} (for simplices and $h^2$-parallelograms) and \cite[Theorems 3.1 and 4.1]{ing:whe:yot:10} (for $h^2$-parallelepipeds), taking into consideration the assumptions (A1)-(A3).
\end{proof}

\subsection{Error estimates}

Next, we will estimate the distance between the elliptic projection $(\tilde{\mathbf{u}}_h,\tilde{p}_h)$ and the semidiscrete solution $(\mathbf{u}_h,p_h)$. The subsequent combination of these bounds with the corresponding results from the previous section yields the convergence for both the velocity and the pressure variables.

Before stating the main result of this section, we will provide some auxiliary bounds that will be used below. In particular, from (\ref{u:elliptic}) and the inverse inequality, it follows \cite{kim:par:tho:whe:07,cho:li:92}
\begin{equation}\label{utildeh:bound}
\|\tilde{\mathbf{u}}_h\|_{L^{\infty}(J;L^{\infty}(\Omega)^d)}\leq C_1.
\end{equation}
Accordingly, using the bound for the time derivative $(p-\tilde{p}_h)_t$ based on (\ref{p:elliptic}) and the inverse inequality, we obtain
\begin{equation}\label{ptildeht:bound}
\|\tilde{p}_{h,t}\|_{L^{\infty}(J;L^{\infty}(\Omega))}\leq C_2.
\end{equation}
Finally, we will require that $p_t\in L^{\infty}(J;L^{\infty}(\Omega))$, and denote $\|p_t\|_{\infty}:=\|p_t\|_{L^{\infty}(J;L^{\infty}(\Omega))}$. Note that, in our setting, this is a reasonable assumption (see, e.g., \cite[Section 3.4]{whe:73}).

Now, we are in position to derive a priori error estimates for both the velocity and pressure variables.

\begin{theorem}
	If $K^{-1}$ and $\rho^{-1}\in W^{1,\infty}_{\mathcal{T}_h}$, then the velocity $\mathbf{u}_h$ and the pressure $p_h$ of the MFMFE method (\ref{mfmfe:method}), with the quadrature rule (\ref{global:quadrature})-(\ref{symmetric:quadrature:physical}), satisfy
	\begin{align}
	&\|\mathbf{u}-\mathbf{u}_h\|_{L^{2}(J;L^2(\Omega)^d)}\leq Ch\left(\|\mathbf{u}\|_{L^{2}(J;H^1(\Omega)^d)}+\|p\|_{L^{2}(J;H^1(\Omega))}
	+\|p_t\|_{L^{2}(J;H^1(\Omega))}\right),\label{theorem:u}\\[1.5ex]
	&\|p-p_h\|_{L^{\infty}(J;L^2(\Omega))}\leq Ch\left(\|\mathbf{u}\|_{L^{\infty}(J;H^1(\Omega)^d)}
	+\|\mathbf{u}\|_{L^{2}(J;H^1(\Omega)^d)}
	+\|p\|_{L^{\infty}(J;H^1(\Omega))}\right.\nonumber\\
	&\hspace*{7.7cm}\left.+\|p\|_{L^{2}(J;H^1(\Omega))}
	+\|p_t\|_{L^{2}(J;H^1(\Omega))}\right)\label{theorem:p}
	\end{align}
	on simplices, $h^2$-parallelograms and $h^2$-parallelepipeds, where $C$ is a positive constant which depends on $\alpha_1$, $\alpha_2$, $\kappa_{\ast}$, $\kappa^{\ast}$, $\gamma_1$, $\gamma_2$, $\eta$, $C_1$, $C_2$ and $\|p_t\|_{\infty}$, but is defined to be independent of $h$.
\end{theorem}

\begin{proof}
	For all $t\in[0,T]$, let us split the velocity and pressure errors in the $L^2(\Omega)$-norm via the triangle inequality, i.e.,
	\begin{align}
	&\|(\mathbf{u}-\mathbf{u}_h)(t)\|\leq
	\|(\mathbf{u}-\tilde{\mathbf{u}}_h)(t)\|
	+\|(\tilde{\mathbf{u}}_h-\mathbf{u}_h)(t)\|,\label{triangle:u}\\[0.5ex]
	&\|(p-p_h)(t)\|\leq
	\|(p-\tilde{p}_h)(t)\|
	+\|(\tilde{p}_h-p_h)(t)\|.\label{triangle:p}
	\end{align}
	Subtracting (\ref{mfmfe:method}) from (\ref{elliptic:projection}), we get the error equations
	\begin{subequations}\label{error:semidiscrete}
		\renewcommand{\theequation}{\theparentequation\alph{equation}}
		\begin{align}
		&(K^{-1}\rho^{-1}(p)\,\tilde{\mathbf{u}}_h
		-K^{-1}\rho^{-1}(p_h)\,\mathbf{u}_h,\mathbf{v})_Q=(\tilde{p}_h-p_h,\nabla\cdot\mathbf{v})\nonumber\\
		&\hspace*{7cm}+((\rho(p)-\rho(p_h))\,\mathbf{g},\mathbf{v}),&&\mathbf{v}\in V_h,\label{error:semidiscrete:a}\\[0.5ex]
		&(\phi\,(\rho(p)_t-\rho(p_h)_t),w)+(\nabla\cdot(\tilde{\mathbf{u}}_h-\mathbf{u}_h),w)=0,&&w\in W_h,\label{error:semidiscrete:b}\\[0.5ex]
		&(\tilde{p}_h(0)-p_h(0),w)=0,&&w\in W_h.\label{error:semidiscrete:c}
		\end{align}
	\end{subequations}
	If we take $\mathbf{v}=\tilde{\mathbf{u}}_h-\mathbf{u}_h$ and $w=\tilde{p}_h-p_h$ in (\ref{error:semidiscrete:a})-(\ref{error:semidiscrete:b}), sum the resulting equations and rearrange some terms, we obtain
	\begin{align}\label{sum:up}
	&(K^{-1}\rho^{-1}(p_h)(\tilde{\mathbf{u}}_h-\mathbf{u}_h),\tilde{\mathbf{u}}_h-\mathbf{u}_h)_Q
	+(\phi\,(\rho(\tilde{p}_h)_t-\rho(p_h)_t),\tilde{p}_h-p_h)\nonumber\\[0.5ex]
	&\hspace*{1.2cm}
	=-((K^{-1}\rho^{-1}(p)-K^{-1}\rho^{-1}(p_h))\,\tilde{\mathbf{u}}_h,\tilde{\mathbf{u}}_h-\mathbf{u}_h)_Q
	+((\rho(p)-\rho(p_h))\,\mathbf{g},\tilde{\mathbf{u}}_h-\mathbf{u}_h)\nonumber\\[0.5ex]
	&\hspace*{9.5cm}
	+(\phi\,(\rho(\tilde{p}_h)_t-\rho(p)_t),\tilde{p}_h-p_h).
	\end{align}
	Based on the coercivity (\ref{coercivity:quadrature}) of the discrete bilinear form $(K^{-1}\rho^{-1}(p)\,\mathbf{q},\mathbf{v})_Q$, the first term on the left in (\ref{sum:up}) is bounded by
	\begin{equation}\label{first:term}
	(K^{-1}\rho^{-1}(p_h)(\tilde{\mathbf{u}}_h-\mathbf{u}_h),\tilde{\mathbf{u}}_h-\mathbf{u}_h)_Q\geq C\,\|\tilde{\mathbf{u}}_h-\mathbf{u}_h\|^2,
	\end{equation}
	where $C$ depends on $\kappa^{\ast}$ and $\gamma_2$. The continuity (\ref{continuity:quadrature}) of such a bilinear form permits to bound the first term on the right in (\ref{sum:up}) as follows
	\begin{align}
	&|(K^{-1}\rho^{-1}(p)-K^{-1}\rho^{-1}(p_h))\,\tilde{\mathbf{u}}_h,\tilde{\mathbf{u}}_h-\mathbf{u}_h)_Q|\leq
	C\,(\|p-p_h\|^2+\|\tilde{\mathbf{u}}_h-\mathbf{u}_h\|^2)
	\end{align}
	where we use (A2), (A3), (\ref{utildeh:bound}), the mean-value theorem and Young's inequality, $ab\leq\frac{1}{2}\left(\epsilon a^2+\frac{1}{\epsilon}b^2\right)$ for all $a,b\geq0$, with $\epsilon=1$. In this case, the constant $C$ depends on $\kappa_{\ast}$, $\gamma_1$ and $C_1$. In the same way, the second term on the right in (\ref{sum:up}) is bounded by
	\begin{equation}
	|((\rho(p)-\rho(p_h))\,\mathbf{g},\tilde{\mathbf{u}}_h-\mathbf{u}_h)|\leq
	C\,(\|p-p_h\|^2+\|\tilde{\mathbf{u}}_h-\mathbf{u}_h\|^2),
	\end{equation}
	using the Cauchy--Schwarz inequality, the mean-value theorem and Young's inequality, where $C$ depends on $\gamma_2$. As for the last term on the right, further applying the chain rule, (A1) and (A3), we get
	\begin{align}
	&(\phi\,(\rho(\tilde{p}_h)_t-\rho(p)_t),\tilde{p}_h-p_h)
	=(\phi\,(\rho'(\tilde{p}_h)-\rho'(p))\,p_t,\tilde{p}_h-p_h)\nonumber\\[0.5ex]
	&\hspace*{1cm}
	+(\phi\,\rho'(\tilde{p}_h)(\tilde{p}_{h,t}-p_t),\tilde{p}_h-p_h) \leq C\,(\|p-\tilde{p}_h\|^2+\|\tilde{p}_h-p_h\|^2+\|(p-\tilde{p}_h)_t\|^2),
	\end{align}
	where $C$ depends on $\alpha_2$, $\eta$ and $\|p_t\|_{\infty}$. It remains to analyze the second term on the left in (\ref{sum:up}). Following \cite{dou:pae:gio:93,par:05,whe:73}, we have
	\begin{align*}
	&(\phi\,(\rho(\tilde{p}_h)_t-\rho(p_h)_t),\tilde{p}_h-p_h)
	=\dfrac{d}{dt}\int_{\Omega}\phi\int_{0}^{\tilde{p}_h-p_h}\rho'(\tilde{p}_h+\xi)\,\xi\,d\xi\,d\mathbf{x}
	\nonumber\\[0.5ex]
	&\hspace*{1cm}
	+\int_{\Omega}\phi\,(\rho'(p_h)-\rho'(\tilde{p}_h))\,\tilde{p}_{h,t}\,(\tilde{p}_h-p_h)\,d\mathbf{x}
	-\int_{\Omega}\phi\int_{0}^{\tilde{p}_h-p_h}\rho''(\tilde{p}_h+\xi)\,\tilde{p}_{h,t}\,\xi\,d\xi\,d\mathbf{x}.
	\end{align*}
	Using (A1), (A3) and (\ref{ptildeht:bound}), we get the bounds
	\begin{align*}
	&\left|\int_{\Omega}\phi\int_{0}^{\tilde{p}_h-p_h}\rho''(\tilde{p}_h+\xi)\,\tilde{p}_{h,t}\,\xi\,d\xi\,d\mathbf{x}\,\right|\leq C\,\|\tilde{p}_h-p_h\|^2,\\[0.5ex]
	&\left|\int_{\Omega}\phi\,(\rho'(p_h)-\rho'(\tilde{p}_h))\,\tilde{p}_{h,t}\,(\tilde{p}_h-p_h)\,d\mathbf{x}\,\right|\leq C\,\|\tilde{p}_h-p_h\|^2,
	\end{align*}
	where $C$ depends on $\alpha_2$, $\eta$ and $C_2$. Hence,
	\begin{align}\label{last:term}
	&(\phi\,(\rho(\tilde{p}_h)_t-\rho(p_h)_t),\tilde{p}_h-p_h)
	\geq\dfrac{d}{dt}\int_{\Omega}\phi\int_{0}^{\tilde{p}_h-p_h}\rho'(\tilde{p}_h+\xi)\,\xi\,d\xi\,d\mathbf{x}-C\,\|\tilde{p}_h-p_h\|^2.
	\end{align}
	In addition, due to (A1) and (A3), it holds
	\begin{align}\label{integral:bound}
	\int_{\Omega}\phi\int_{0}^{\tilde{p}_h-p_h}\rho'(\tilde{p}_h+\xi)\,\xi\,d\xi\,d\mathbf{x}
	\geq C\,\|\tilde{p}_h-p_h\|^2.
	\end{align}
	In this case, $C$ depends on $\alpha_1$ and $\gamma_1$.	Inserting (\ref{first:term})-(\ref{last:term}) into (\ref{sum:up}), we obtain
	\begin{align*}
	&\dfrac{d}{dt}\int_{\Omega}\phi\int_{0}^{\tilde{p}_h-p_h}\rho'(\tilde{p}_h+\xi)\,\xi\,d\xi\,d\mathbf{x}
	+C\,\|\tilde{\mathbf{u}}_h-\mathbf{u}_h\|^2\nonumber\\[0.5ex]
	&\hspace*{5cm}
	\leq C\,(\|p-\tilde{p}_h\|^2+\|\tilde{p}_h-p_h\|^2+\|(p-\tilde{p}_h)_t\|^2).
	\end{align*}
	Integrating with respect to $t$, using (\ref{integral:bound}) and Gronwall's lemma, yields
	\begin{align*}
	&\|(\tilde{p}_h-p_h)(t)\|^2+\int_0^t\|(\tilde{\mathbf{u}}_h-\mathbf{u}_h)(\tau)\|^2\,d\tau
	\leq C\int_0^t(\|(p-\tilde{p}_h)(\tau)\|^2+\|(p-\tilde{p}_h)_t(\tau)\|^2)\,d\tau,
	\end{align*}
	for all $t\in[0,T]$. Note that the initial condition implies $(\tilde{p}_h-p_h)(0)=0$ and the right-hand side is bounded by (\ref{p:elliptic}). If we consider $t=T$ in the previous inequality, together with (\ref{u:elliptic}) and (\ref{triangle:u}), we get the bound (\ref{theorem:u}). In turn, if we take the supremum over all $t$, and consider the inequalities (\ref{p:elliptic}) and (\ref{triangle:p}), we obtain the bound (\ref{theorem:p}).
\end{proof}

\section{The fully discrete formulation}\label{section:fully:discrete}

In order to derive a fully discrete formulation of the MFMFE method (\ref{mfmfe:method}), we consider a time integration using the backward Euler method. For simplicity, given $N\in\mathbb{N}$, we consider an equidistant time grid $0=t_0<t_1<\ldots<t_{N}=T$, where $t_n=n\tau$ and $\tau=T/N$. The fully discrete approximation to \eqref{ibvp} is given by: \textit{Find $(\mathbf{{u}}_h^{n+1},p_h^{n+1})\in V_h\times W_h$ such that, for $n=0,1,\ldots,N-1$,}
\begin{subequations}\label{mfmfe}
	\begin{align}
	&\left(K^{-1}\rho^{-1}(p_{h}^{n+1})\mathbf{u}_h^{n+1},\mathbf{v}\right)_Q=\left(p_h^{n+1},\nabla\cdot\mathbf{v}\right)+\left(\rho(p_{h}^{n+1})\,\mathbf{g},\mathbf{v}\right),&&\mathbf{v}\in V_h,\label{mfmfe:b}\\[0.5ex]
	&\left(\phi\,\rho(p_{h}^{n+1}),w\right)+\left(\tau\,\nabla\cdot\mathbf{u}_h^{n+1},w\right)=\left(\phi\,\rho(p_{h}^{n})+\tau f^{n+1},w\right),&&{w}\in W_h,\label{mfmfe:a}\\[0.5ex]
	&\,\,p_h^0=p_h(0),\label{mfmfe:c}
	\end{align}
\end{subequations}
where $f^{n+1}=f(\cdot,t_{n+1})$.

Let $N_{\ell}$, $N_e$ and $N_v$ denote the number of faces (edges), elements and vertices in $\mathcal{T}_h$. We denote $L=nN_{\ell}$, where $n$ is the number of vertices per face (edge). That is, $L$ and $N_e$ are the number of degrees of freedom of the totally discrete velocity and pressure functions, respectively. In this context, let $\{\mathbf{v}_i\}_{i=1}^{L}$ and $\{w_i\}_{i=1}^{N_e}$ be the finite element basis functions for $V_h$ and $W_h$, respectively.
Therefore, the unknowns of problem (\ref{mfmfe}) can be expressed as
$$\mathbf{u}_h^{n+1}=\sum_{i=1}^{L}U_{h,i}^{n+1}\mathbf{v}_i, \qquad p_h^{n+1}=\sum_{i=1}^{N_e}P_{h,i}^{n+1}w_i.$$

If we group the velocity unknowns per vertices, we can define a vector of unknown velocity coefficients
$$U_{h}^{n+1}=[\tilde{U}_{h,1}^{n+1},\tilde{U}_{h,2}^{n+1},\ldots,\tilde{U}_{h,N_v}^{n+1}]^T\in\mathbb{R}^{L},$$ 
where $\tilde{U}_{h,1}^{n+1}\in\mathbb{R}^{\ell_i}$, $\ell_i$ being the number of faces (edges) that share the $i$-th vertex point. The component of $U_{h}^{n+1}$ associated to the face (edge) $e_j$ at vertex $\mathbf{r}_i$ is given by the volumetric flux $(\mathbf{u}_h^{n+1}\cdot\mathbf{n}_{e_j})(\mathbf{r}_i)|e_j|$, where $|e_j|$ denotes the area (length) of $e_j$, for $j=1,2,\ldots,\ell_i$. 
On the other hand, the pressure unknowns can also be grouped into the vector
$$P_h^{n+1}=[P_{h,1}^{n+1}, P_{h,2}^{n+1},\ldots, P_{h,N_e}^{n+1}]^T\in\mathbb{R}^{N_e},$$
where $P_{h,i}^{n+1}=p_h^{n+1}(\mathbf{x}_{c,i})$, $\mathbf{x}_{c,i}$ being the coordinate vector of the center of mass of the $i$-th element.

Considering the lighter notation $P_i:=P_{h,i}^{n+1}$ and $U_j:=U_{h,j}^{n+1}$, for $i=1,2,\ldots,N_e$ and $j=1,2,\ldots,L$, the nonlinear problem (\ref{mfmfe}) may be rewritten as the following system of nonlinear  equations
\begin{subequations}
	\begin{align*}
	&F_j(U_h^{n+1},P_h^{n+1})=\left(\rho_0^{-1}e^{-c_f\sum_{i=1}^{N_e}P_i\,w_i}K^{-1}\sum_{i=1}^{L}U_i\mathbf{v}_i,\mathbf{v}_j\right)_Q-\left(\sum_{i=1}^{N_e}P_i\,w_i,\nabla\cdot\mathbf{v}_j\right)\nonumber\\
	&\hspace*{5.5cm}-\left(\rho_0\,e^{c_f\sum_{i=1}^{N_e}P_i\,w_i}\,\mathbf{g},\mathbf{v}_j\right)=0, \qquad j=1,2,\ldots,L,\\[1ex]
	&G_j(U_h^{n+1},P_h^{n+1})=-\left(\phi\,\rho_0\,e^{c_f\sum_{i=1}^{N_e}P_i\,w_i},w_j\right)-\left(\tau\nabla\cdot\sum_{i=1}^{2N_{\ell}}U_i\mathbf{v}_i,w_j\right)\nonumber\\[1ex]
	&\hspace*{5.5cm}+(\phi\,\rho(p_{h}^{n})+\tau f^{n+1},w_j)=0,\qquad j=1,2,\ldots,N_e,
	\end{align*}
\end{subequations}
where $\rho_0=\rho_{\mathrm{ref}}\,e^{-c_f\,p_{\mathrm{ref}}}$. Thus, at time level $n+1$, the fully discrete problem (\ref{mfmfe}) is equivalent to the nonlinear system of $L+N_e$ residual equations
\begin{equation}
\label{non:lin:sys}
\begin{array}{l}
F(U_h^{n+1},P_h^{n+1})=0,\\[0.8ex]
G(U_h^{n+1},P_h^{n+1})=0,
\end{array}
\end{equation}
for the unknown vector $[U_h^{n+1},P_h^{n+1}]^T$ of the same size.

\section{Solution of the residual equations}\label{section:residual:eqs}

Classical iterative methods for solving nonlinear systems of equations include the Picard iteration and Newton-type methods \cite{kel:95}. In the framework of variably saturated flow problems, this type of linearization schemes was successfully applied to the solution of the resulting systems from various discretizations of Richards' equation \cite{cel:bou:zar:90,pan:put:94,ber:put:99}. A new linearization technique, the so-called $L$--scheme, was subsequently proposed for a class of nonlinear and degenerate parabolic problems, including Richards' equation (see \cite{pop:rad:kna:04} and references therein). This fixed-point linearization scheme was proven to be a valuable alternative to Picard or Newton methods for degenerate problems \cite{lis:rad:16,rad:nor:pop:kum:15}. More recently, suitable combinations of the preceding techniques, such as the Picard/Newton method or the $L$--scheme/Newton method, have also been proposed and analyzed in the literature (see \cite{lis:rad:16} and references therein).

Despite the relatively high CPU cost per-iteration of Newton method, its second-order convergence when the initial guess is close enough to the solution makes it a powerful tool for solving nonlinear systems arising from the discretization of non-degenerate parabolic equations. In this setting, we propose an efficient solution strategy for the nonlinear system (\ref{non:lin:sys}), based on a Newton-type method. For that purpose, we compute the partial derivatives of the residual equations with respect to each unknown, and introduce the following notations
\begin{subequations}\label{res:eq}
	\begin{align}
	&A_{ji}^{n+1}=\frac{\partial F_j}{\partial U_i}(U_h^{n+1},P_h^{n+1})=(\rho^{-1}(p_h^{n+1})K^{-1}\mathbf{v}_i,\mathbf{v}_j)_Q,\hspace*{1.2cm} i,\,j=1,2,\ldots,L,\label{res:eq:a}\\[1ex]
	&\tilde{B}_{ji}^{n+1}=\frac{\partial F_j}{\partial P_i}(U_h^{n+1},P_h^{n+1}t)=-(c_fw_i\,\rho^{-1}(p_h^{n+1})K^{-1}\mathbf{u}_h^{n+1},\mathbf{v}_j)_Q-(w_i,\nabla\cdot\mathbf{v}_j)\nonumber\\
	&\hspace*{3cm}-(c_fw_i\rho(p_h^{n+1})\,\mathbf{g},\mathbf{v}_j),\qquad i=1,2,\ldots,N_e,\ j=1,2,\ldots,L,\label{res:eq:b}	\\[1ex]
	&C_{ji}=\frac{\partial G_j}{\partial U_i}(U_h^{n+1},P_h^{n+1})=-(\tau \nabla\cdot\mathbf{v}_i,w_j),\quad  i=1,2,\ldots,L,\ j=1,2,\ldots,N_e,\label{res:eq:c}\\[1ex]
	&D_{ji}^{n+1}=\frac{\partial G_j}{\partial P_i}(U_h^{n+1},P_h^{n+1})=-(\phi \,c_fw_i\rho(p_h^{n+1}),w_j),\hspace*{1.3cm} i,\,j=1,2,\ldots,N_e.\label{res:eq:d}
	\end{align}
\end{subequations}
Following \cite{gan:jun:pen:whe:yot:14}, since we are dealing with slightly compressible single-phase flow (i.e., the compressibility constant $c_f$ is small), we drop the two terms containing $c_f$ in (\ref{res:eq:b}). In this way, we obtain a nearly symmetric Jacobian matrix 
\[\begin{bmatrix}
A^{n+1} & B \\
C & D^{n+1}
\end{bmatrix},
\]
where $A^{n+1}$ is a symmetric and positive definite block-diagonal matrix, as we will see below, $D^{n+1}$ is a diagonal matrix and $C=\tau\,B^T$. Hence, given $[U_{h}^{n+1,k},P_{h}^{n+1,k}]^T$, the linear system to solve at the $(k+1)$-th inexact Newton iteration is
\begin{equation}
\label{newton:it}
\begin{bmatrix}
A & B \\
\tau B^T & D
\end{bmatrix}
\begin{bmatrix}
\Delta U_h^{n+1,k+1}\\
\Delta P_h^{n+1,k+1}
\end{bmatrix}
=-\begin{bmatrix}
F(U_{h}^{n+1,k},P_{h}^{n+1,k})\\
G(U_{h}^{n+1,k},P_{h}^{n+1,k})
\end{bmatrix},
\end{equation}
where the notations $A=A^{n+1,k}$, $D=D^{n+1,k}$ have been adopted. Note that $c_f$ is not assumed to be zero in the computation of the residual vector (right-hand side). Once the system (\ref{newton:it}) has been solved, $[U_{h}^{n+1,k+1},P_{h}^{n+1,k+1}]^T$ is computed as usual
\begin{equation*}
\begin{bmatrix}
U_h^{n+1,k+1}\\
P_h^{n+1,k+1}
\end{bmatrix}
=\begin{bmatrix}
U_h^{n+1,k}\\
P_h^{n+1,k}
\end{bmatrix}
+\begin{bmatrix}
\Delta U_h^{n+1,k+1}\\
\Delta P_h^{n+1,k+1}
\end{bmatrix}.
\end{equation*}
According to the theory on inexact Newton methods \cite{kel:95}, when convergence is achieved, the solution to this nonlinear system will coincide with that provided by the full Newton method.

Finally, we will describe how to eliminate the velocity unknowns from each Newton iteration (\ref{newton:it}). It is significant to note that the use of the quadrature rule defined by (\ref{global:quadrature})-(\ref{symmetric:quadrature:physical}) permits to decouple the velocity degrees of freedom associated to a vertex from the rest of them. As a consequence, the matrix $A$ has a block-diagonal structure of the form $A=\mathrm{diag}(A_1,A_2,\ldots,A_{N_v})$, where each block $A_i\in\mathbb{R}^{\ell_i\times\ell_i}$ is related to the velocity unknowns associated to the $i$-th mesh vertex and can be easily inverted. In this framework, we can write
$$
\Delta U_{h}^{n+1,k+1}=-A^{-1}(B\Delta P_{h}^{n+1,k+1}+F(U_{h}^{n+1,k},P_{h}^{n+1,k})).
$$ 
The previous formula permits us to express the velocity unknowns associated to each corner in terms of the pressure unknowns located at the centers of the elements that share that corner. Thus, each Newton-type iteration (\ref{newton:it}) can be reduced to the solution of a cell-centered linear system for the pressure increments
$$
(\tau B^{T}A^{-1}B-D)\Delta P_{h}^{n+1,k+1}=G(U_{h}^{n+1,k},P_{h}^{n+1,k})
-\tau B^TA^{-1}F(U_{h}^{n+1,k},P_{h}^{n+1,k}).
$$
The previous system matrix is symmetric and positive definite and represents a 27-point or 9-point stencil on logically cubic or rectangular grids, respectively.

\section{Numerical experiments}\label{section:numer:examples}

This section contains some numerical experiments that show the performance of the  method when applied to slightly compressible flow problems in different scenarios. In the sequel, tensor $K$ is considered to be defined as $K=\mu^{-1}\hat{K}$, where $\mu$ is the fluid viscosity and $\hat{K}$ is the rock permeability tensor.

\subsection{Smooth solution test}

Let us consider a no gravity two-dimensional initial-boundary value problem of type (\ref{ibvp}), where $\Omega=(0,1)^2$, $\Gamma_D=\partial\Omega$, $T=2$ and tensor $\hat{K}$ is given by
$$
\hat{K}:=\hat{K}(x,y)=\begin{bmatrix}
4+(x+2)^2+y^2 & 1+x\,y \\
1+x\,y & 2
\end{bmatrix}.
$$
Data functions $f(\mathbf{x},t)$ and $p_0(\mathbf{x})$ are defined in such a way that the exact solution is $p(\mathbf{x},t)=t\sin(3\pi x)^2\sin(3\pi y)^2$. On the other hand, the parameter values are chosen to be  $c_f=4\times10^{-5}$, $\mu=2$, $\phi=0$.$2$, $\rho_{\mathrm{ref}}=1$ and $p_{\mathrm{ref}}=0$.

The spatial domain is discretized by means of three types of quadrilateral meshes consisting of $N\times N$ elements. The first one is a family of smooth meshes composed of $h^2$-parallelograms, where $h=1/N$. It is defined as the following $C^{\infty}$-map of successively refined uniform meshes on the unit square
$$
\begin{array}{l}x=\hat{x}+\frac{3}{50}\sin(2\pi\hat{x})\sin(2\pi\hat{y}),\\[0.5ex]
y=\hat{y}-\frac{1}{20}\sin(2\pi\hat{x})\sin(2\pi\hat{y}).\end{array}
$$
An illustration of this type of meshes is given in Figure \ref{fig:meshes}(a). Next, we consider a set of Kershaw-type meshes \cite{Ker:81} which contain certain highly skewed zones, as shown in Figure \ref{fig:meshes}(b). 
Finally, we consider a family of randomly $h$-perturbed meshes consisting of highly distorted quadrilaterals. Each of these meshes is
generated by perturbing the vertices of a uniform mesh by a distance
of size $O(h)$ in a random direction, see Figure \ref{fig:meshes}(c). More specifically, the vertices of the randomly perturbed mesh can be defined as
$$
\begin{array}{l}
x_{i,j}=\hat{x}_{i,j}-\frac{5}{4}h+\frac{\sqrt{2}}{3}hr_x^{i,j},\\[1ex]
y_{i,j}=\hat{y}_{i,j}-\frac{5}{4}h+\frac{\sqrt{2}}{3}hr_y^{i,j},
\end{array}
$$
where $r_x^{i,j}$ and $r_y^{i,j}$ are pseudo-random numbers uniformly distributed in the interval $(0,1)$.

\begin{figure}[t]
	\begin{center}
		\begin{minipage}[t]{0.33\textwidth}
			\begin{center}\includegraphics[scale=0.16]{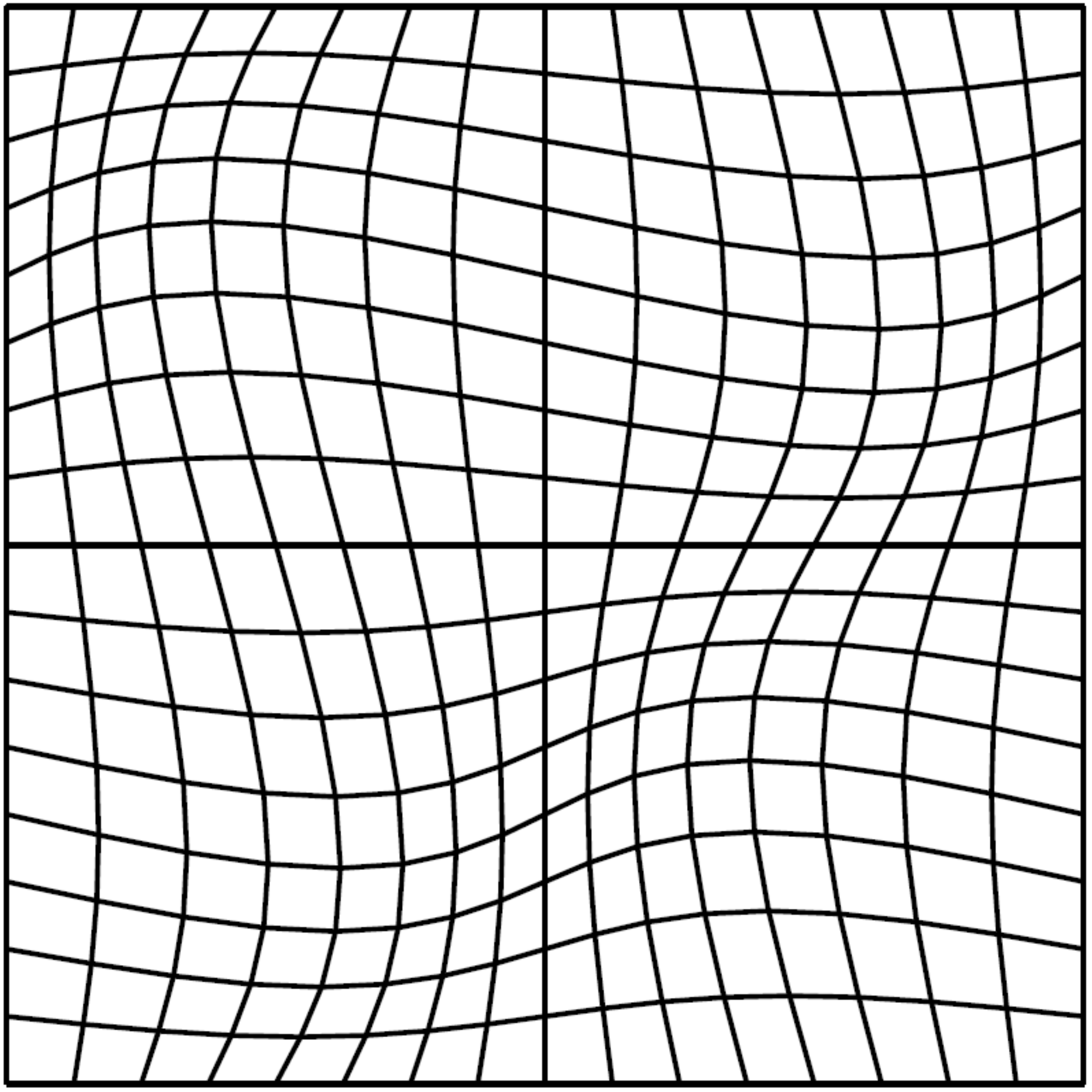}\\{\footnotesize (a) Smooth mesh}\end{center}
		\end{minipage}
		\hspace*{-0.25cm}
		\begin{minipage}[t]{0.33\textwidth}\begin{center}\includegraphics[scale=0.16]{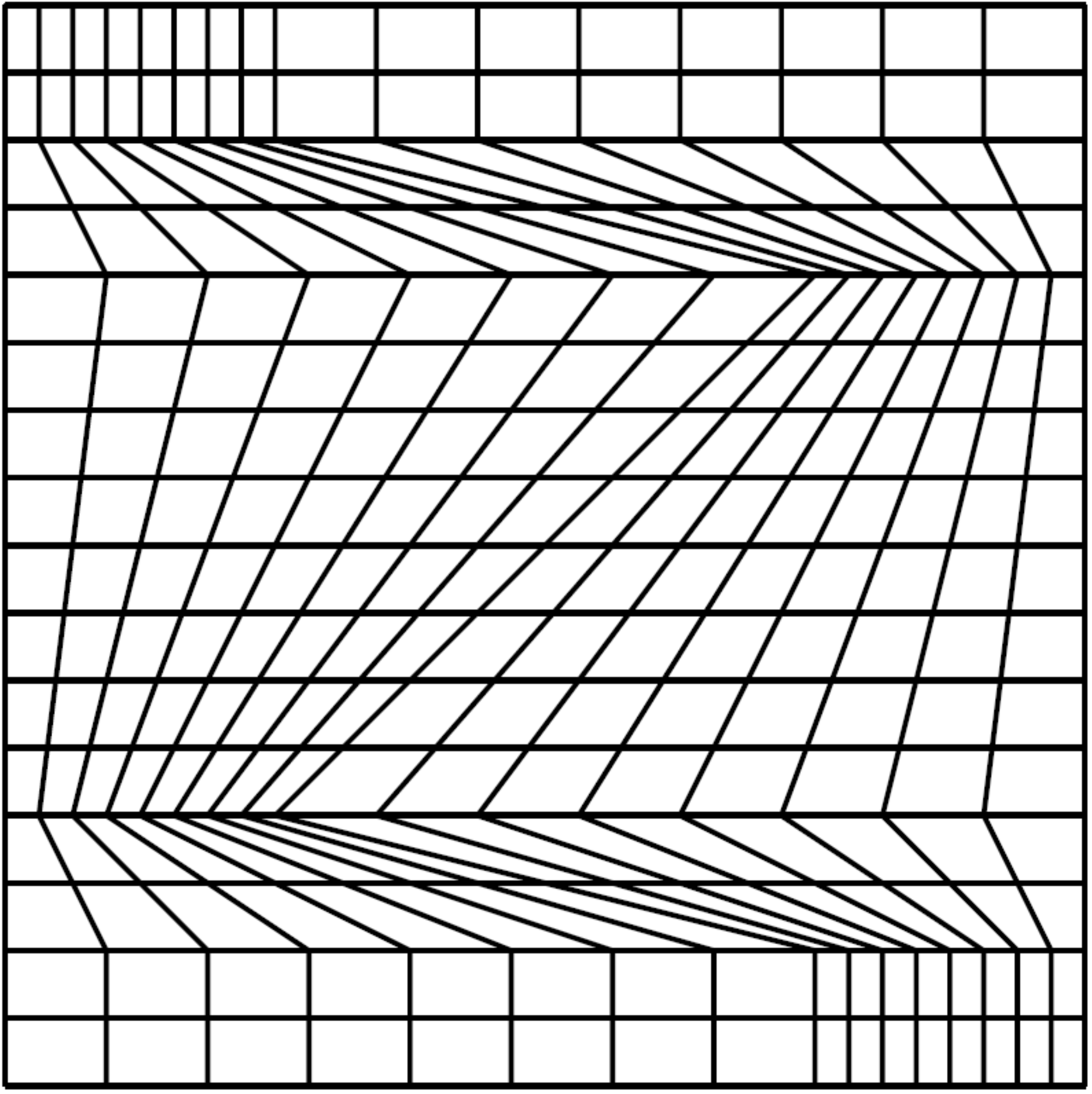}\\{\footnotesize (b) Kershaw mesh}\end{center}
		\end{minipage}
		\hspace*{-0.25cm}
		\begin{minipage}[t]{0.33\textwidth}\begin{center}\includegraphics[scale=0.16]{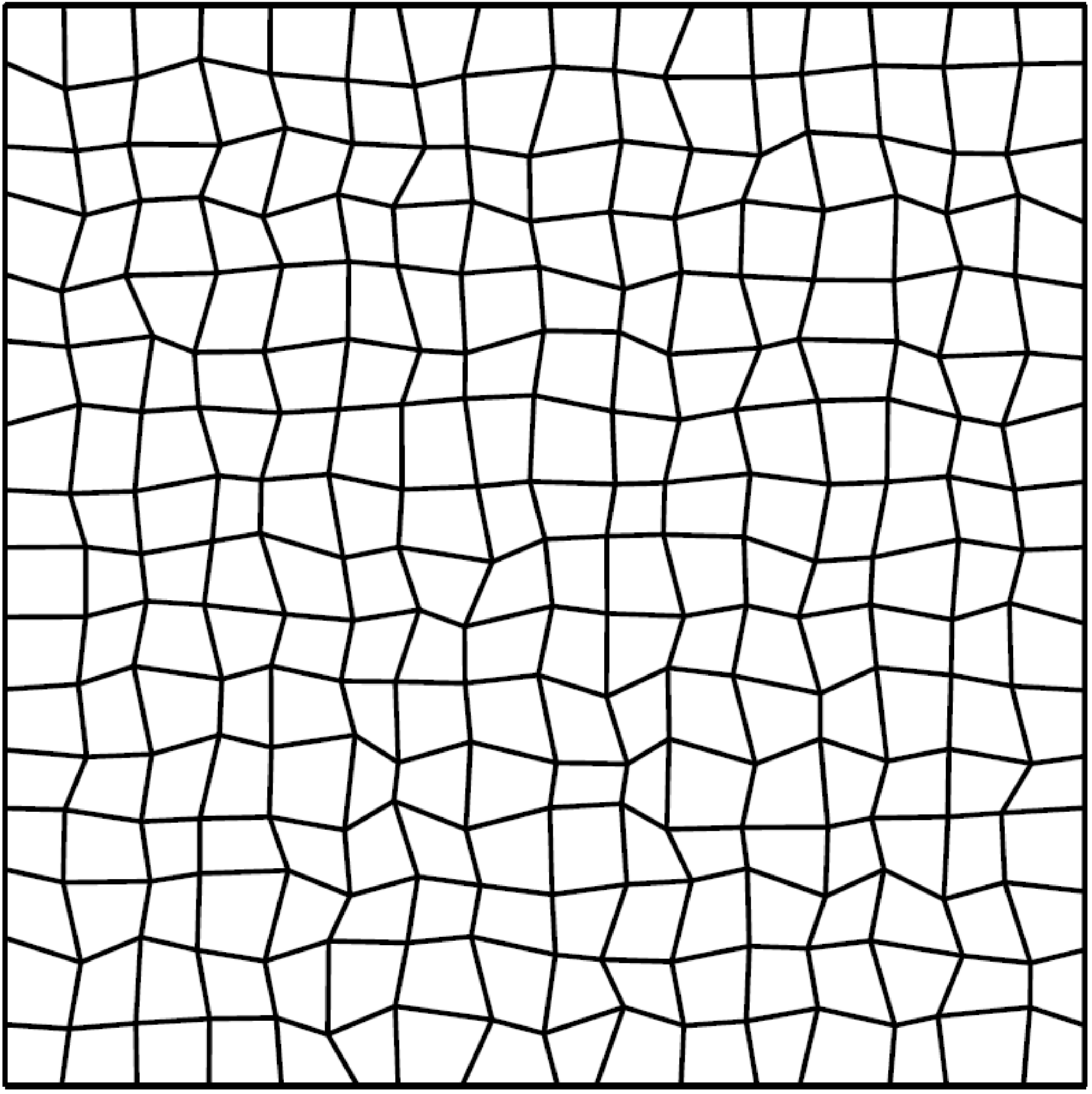}\\{\footnotesize (c) Randomly $h$-perturbed mesh}\end{center}
		\end{minipage}
	\end{center}\caption{Quadrilateral meshes for the numerical tests.}\label{fig:meshes}
\end{figure}

Next, we test the spatial convergence of the MFMFE method on the previous meshes. Taking into account that the time discretization error is negligible for this problem, we set the time step to a fixed value $\tau=0.1$. The spatial errors will be measured combining the $\ell^{\infty}$-norm in time with various norms in space, i.e.,
\begin{subequations}\label{norms}
	\begin{align}
	E^p_{h,\tau}&:=\|p-p_h\|_{\ell^{\infty}(L^2)}=\textstyle\max_{0\leq n\leq N}\|p(t_{n+1})-p_h^{n+1}\|,\label{norm:p:1}\\
	{E}^{\mathbf{u}}_{h,\tau}&:=\|\Pi_h \mathbf{u}-\mathbf{u}_h\|_{\ell^{\infty}(L^2)}=\textstyle\max_{0\leq n\leq N}\|\Pi_h \mathbf{u}(t_{n+1})-\mathbf{u}_h^{n+1}\|,\label{norm:u:1}\\
	\widehat{E}^p_{h,\tau}&:=\|r_hp-P_h\|_{\ell^{\infty}(\ell^2)}=\textstyle\max_{0\leq n\leq N}\|r_hp(t_{n+1})-P_h^{n+1}\|_{\ell^2},\label{norm:p:2}\\
	\widehat{E}^{\mathbf{u}}_{h,\tau}&:=\|\mathbf{u}-\mathbf{u}_h\|_{\ell^{\infty}(\mathcal{F}_h)}=\textstyle\textstyle\max_{0\leq n\leq N}\|\mathbf{u}(t_{n+1})-\mathbf{u}_h^{n+1}\|_{\mathcal{F}_h}.\label{norm:u:2}
	\end{align}
\end{subequations}
The integrals involved in the pressure errors (\ref{norm:p:1}) are approximated element-wise by a 9-point Gaussian quadrature formula. In turn, those involved in the velocity errors (\ref{norm:u:1}) are approximated by the trapezoidal quadrature rule. According to the theoretical results derived in Section \ref{section:analysis}, both errors (\ref{norm:p:1}) and (\ref{norm:u:1}) should show a fist-order convergent behaviour when the method is applied on $\mathcal{O}(h^2)$-parallelogram grids. Moreover, in order to check the convergence of the numerical pressure at the cell centers, we compute the errors (\ref{norm:p:2}), where $r_h$ denotes the restriction to the center of the cells. Finally, according to \cite{whe:xue:yot:12,arr:por:yot:14}, we shall study the convergence of the numerical velocity on the element edges. To this end, we obtain results in the edge-based norm (\ref{norm:u:2}), where $\|\cdot\|_{\mathcal{F}_h}$ is defined to be 
\[
\|\mathbf{v}\|^2_{\mathcal{F}_h}=\sum_{E\in\mathcal{T}_h}\sum_{e\in\partial E}\frac{|E|}{|e|}\,\|\mathbf{v}\cdot\mathbf{n}_e\|_e^2.
\]
Note that the integrals involved in the face errors (\ref{norm:u:2}) are computed by a high-order Gaussian quadrature rule.

Tables \ref{table:space:smooth} and \ref{table:space:kershaw} show the global errors and numerical orders of convergence in space of the method when applied on two families of $h^2$-parallelogram grids: the smooth meshes and the Kershaw-type meshes, respectively. As the theory predicts, we observe first-order convergence for the pressure and the velocity. Moreover, for both types of meshes, we obtain first-order convergence for the velocity on the element edges and second-order superconvergence for the pressure at the cell centers.

\begin{table}[t]
	\renewcommand{\arraystretch}{1.2}
	\centering{\small
		\begin{tabular}{l|cccccccc}
			\hline\\[-3ex]			
			$h$ & $E^p_{h,\tau}$ & Rate & $\hat{E}^p_{h,\tau}$ & Rate & $E^{\mathbf{u}}_{h,\tau}$ & Rate & $\hat{E}^{\mathbf{u}}_{h,\tau}$ & Rate \\
			\hhline{|---------|}
			$h_0$ & 2.233e-01 & -- & 7.994e-02 & -- & 1.317e{\small+}01& -- & 1.164e{\small +}01& -- \\
			$h_0/2$& 1.067e-01 & 1.065 & 2.023e-02 & 1.982 & 6.735e{\small +}00& 0.968 & 5.634e{\small +}00& 1.047 \\
			$h_0/2^2$& 5.258e-02 & 1.021 & 5.094e-03 & 1.990 & 3.390e{\small +}00& 0.990 & 2.786e{\small +}00& 1.016 \\
			$h_0/2^3$& 2.621e-02 & 1.004 & 1.277e-03 & 1.996 & 1.698e{\small +}00& 0.997 & 1.388e{\small +}00& 1.005 \\
			$h_0/2^4$& 1.309e-02 & 1.002 &  3.194e-04 & 1.999 & 8.491e-01& 1.000 &  6.936e-01& 1.001 \\
			\hline			
		\end{tabular}\caption{Global errors and numerical orders of convergence in space on the smooth meshes ($\tau=0.1$, $h_0=2^{-4}$).}\label{table:space:smooth}
	}
\end{table}

\begin{table}[t]	
	\renewcommand{\arraystretch}{1.2}
	\centering{\small
		\begin{tabular}{l|cccccccc}
			\hline\\[-3ex]			
			$h$ & $E^p_{h,\tau}$ & Rate & $\hat{E}^p_{h,\tau}$ & Rate & $E^{\mathbf{u}}_{h,\tau}$ & Rate & $\hat{E}^{\mathbf{u}}_{h,\tau}$ & Rate \\
			\hhline{|---------|}
			$h_0$ & 2.913e-01 & -- & 1.667e-01 & -- & 2.008e{\small+}01& -- & 1.922e{\small +}01& -- \\
			$h_0/2$& 1.572e-01 & 0.890 & 6.205e-02 & 1.426 & 1.149e{\small +}01& 0.805 & 1.053e{\small +}01& 0.868 \\
			$h_0/2^2$& 7.917e-02 & 0.990 & 1.631e-02 & 1.928 & 5.316e{\small +}00& 1.112 & 4.673e{\small +}00& 1.172 \\
			$h_0/2^3$& 3.945e-02 & 1.005 & 4.106e-03 & 1.990 & 2.502e{\small +}00& 1.087 & 2.098e{\small +}00& 1.155 \\
			$h_0/2^4$& 1.970e-02 & 1.002 &  1.030e-03 & 1.995 & 1.227e{\small +}00& 1.028 &  1.010e{\small +}00& 1.055 \\
			\hline			
		\end{tabular}\caption{Global errors and numerical orders of convergence in space on the Kershaw meshes ($\tau=0.1$, $h_0=2^{-4}$).}\label{table:space:kershaw}
	}
\end{table}

Table \ref{table:space:random:sym} shows the global errors and the numerical orders of convergence on the randomly $h$-perturbed grids. The numerical results show that the convergence of the velocity and the pressure of the MFMFE method deteriorates on these highly distorted grids. Similar numerical results are reported in \cite{whe:xue:yot:12} for the MFMFE method applied to the numerical solution of incompressible problems on randomly $h$-perturbed meshes.

\begin{table}[t]	
	\renewcommand{\arraystretch}{1.2}
	\centering{\small
		\begin{tabular}{l|cccccccc}
			\hline\\[-3ex]			
			$h$ & $E^p_{h,\tau}$ & Rate & $\hat{E}^p_{h,\tau}$ & Rate & $E^{\mathbf{u}}_{h,\tau}$ & Rate & $\hat{E}^{\mathbf{u}}_{h,\tau}$ & Rate \\
			\hhline{|---------|}
			$h_0$ & 2.216e-01 & -- & 7.709e-02 & -- & 1.261e{\small+}01& -- & 1.142e{\small +}01& -- \\
			$h_0/2$& 1.054e-01 & 1.072 & 2.331e-02 & 1.726 & 7.362e{\small +}00& 0.776 & 6.199e{\small +}00& 0.881 \\
			$h_0/2^2$& 5.285e-02 & 0.996 & 1.200e-02 & 0.958 & 4.919e{\small +}00& 0.582 & 4.127e{\small +}00& 0.587 \\
			$h_0/2^3$& 2.747e-02 & 0.944 & 9.733e-03 & 0.302 & 4.034e{\small +}00& 0.286 & 3.394e{\small +}00& 0.282 \\
			$h_0/2^4$& 1.601e-02 & 0.779 &  9.563e-03 & 0.025 & 3.861e{\small +}00& 0.063 &  3.263e{\small +}00& 0.057\\
			\hline			
		\end{tabular}}\caption{Global errors and numerical orders of convergence in space on the randomly $h$-perturbed meshes ($\tau=0.1$, $h_0=2^{-4}$).}\label{table:space:random:sym}
\end{table}

Following the ideas in \cite{kla:win:06,whe:xue:yot:12}, if the mesh is composed of highly distorted quadrilaterals or hexahedra, it is convenient to define a non-symmetric quadrature rule on each element $E\in\mathcal{T}_h$
\begin{equation*}
(K^{-1}\rho^{-1}(p)\,\mathbf{q},\mathbf{v})_{Q,E}:=(\tilde{\mathcal{K}}_E^{-1}\hat{\mathbf{q}},\hat{\mathbf{v}})_{\hat{Q},\hat{E}}:=
\frac{|\hat{E}|}{n_v}\sum_{i=1}^{n_v}\tilde{\mathcal{K}}_E^{-1}(\hat{\mathbf{r}}_i)\hat{\mathbf{q}}(\hat{\mathbf{r}}_i)\cdot\hat{\mathbf{v}}(\hat{\mathbf{r}}_i),
\end{equation*}
where
$
\tilde{\mathcal{K}}_E^{-1}(\hat{\mathbf{x}})=J_E^{-1}(\hat{\mathbf{x}})DF_E^T(\hat{\mathbf{x}}_{c})\overline{K}_E^{-1}\overline{\rho}_E^{-1}DF_E(\hat{\mathbf{x}})$.
Here $\overline{K}_E$ is a constant matrix such that $(\overline{K}_E)_{ij}$ is the mean value of $(K)_{ij}$ on $E$, $(\overline{K}_E)_{ij}$ and $(K)_{ij}$ being the elements on the $i$-th row and $j$-th column of matrices $\overline{K}_E$ and $K$, respectively. Similarly, $\overline{\rho}_E$ denotes the mean value of $\rho$ on $E$. Furthermore, $\hat{\mathbf{x}}_{c}$ denotes the center of mass of $\hat{E}$.
The transformation back to the physical element $E$ yields
\begin{equation}\label{nonsymmetric:quadrature:physical}
(K^{-1}\rho^{-1}(p)\,\mathbf{q},\mathbf{v})_{Q,E}=\frac{1}{n_v}\sum_{i=1}^{n_v}J_E(\hat{\mathbf{r}}_i)DF_E^{-T}(\mathbf{r}_i)DF_E^T(\hat{\mathbf{x}}_{c})
\overline{K}_{E}^{-1}\overline{\rho}_E^{-1}\mathbf{q}(\mathbf{r}_i)\cdot\mathbf{v}(\mathbf{r}_i).
\end{equation}
Given this local formula, the global quadrature rule is derived from (\ref{global:quadrature}).

Table \ref{table:space:random:non:sym} show the numerical results obtained when considering the non-symmetric MFMFE method derived by using the previous non-symmetric quadrature rule. Similarly to the case of incompressible flows \cite{whe:xue:yot:12,arr:por:yot:14}, the non-symmetric MFMFE method has first-order convergence for both the velocity and the pressure. Moreover, it shows first-order convergence for the velocity on the element edges and second-order superconvergence for the pressure at the cell centers. 

\begin{table}[t]
	\renewcommand{\arraystretch}{1.2}
	\centering{\small
		\begin{tabular}{l|cccccccc}
			\hline\\[-3ex]			
			$h$ & $E^p_{h,\tau}$ & Rate & $\hat{E}^p_{h,\tau}$ & Rate & $E^{\mathbf{u}}_{h,\tau}$ & Rate & $\hat{E}^{\mathbf{u}}_{h,\tau}$ & Rate \\
			\hhline{|---------|}
			$h_0$ & 2.226e-01 & -- & 7.797e-02 & -- & 1.289e{\small+}01& -- & 1.155e{\small +}01& -- \\
			$h_0/2$& 1.051e-01 & 1.083 & 1.915e-02 & 2.026 & 6.855e{\small +}00& 0.911 & 5.686e{\small +}00& 1.022 \\
			$h_0/2^2$& 5.201e-02 & 1.015 & 5.035e-03 & 1.927 & 3.409e{\small +}00& 1.008 & 2.782e{\small +}00& 1.031 \\
			$h_0/2^3$& 2.587e-02 & 1.008 & 1.220e-03 & 2.045 & 1.720e{\small +}00& 0.987 & 1.394e{\small +}00& 0.997 \\
			$h_0/2^4$& 1.293e-02 & 1.001 &  3.123e-04 & 1.966 & 8.584e-01& 1.003 &  6.954e-01& 1.003 \\
			\hline
	\end{tabular}}\caption{Global errors and numerical orders of convergence in space for the non-symmetric MFMFE method on the randomly $h$-perturbed meshes ($\tau=0.1$, $h_0=2^{-4}$).}\label{table:space:random:non:sym}
\end{table}

\subsection{Quarter five-spot problem}

A five-spot pattern is a standard configuration in petroleum engineering in which four input or injection wells are located at the corners of a square while the production well sits in the center. A certain fluid, which is normally water, steam or gas, is injected simultaneously through the four input wells to displace the oil towards the central production well. Due to the symmetry of the problem, it is usual to consider only a quarter five-spot pattern on the unit square with injection and production wells located at (0,0) and (1,1), respectively. Following \cite{cha:pie:for:17}, we consider the following source term model
\begin{align*}f(x,y)=200\,&(\tanh(200(0.025-(x^2+y^2)^{1/2}))\\&-\tanh(200(0.025-((x-1)^2+(y-1)^2)^{1/2})))
\end{align*}
and  impose homogeneous Dirichlet and Neumann boundary conditions on
$$
\Gamma_D=\{(x,y)\in\partial \Omega: x=1 \hbox{ and } y\leq3/4\}\cup  \{(x,y)\in\partial \Omega: y=1 \hbox{ and } x\leq3/4\}
$$
and $\Gamma_N=\partial \Omega\backslash\Gamma_D$, respectively.

Next, we consider a regular mesh with 128 cells in each direction and a constant time step $\tau=5\times 10^{-3}$. We choose the same parameter values as in the previous example and we define the initial condition to be $p_0(x,y)=(1-3 x^2 + 2 x^3) (1 - 3 y^2 + 2 y^3).$ In this framework, we consider different permeability tensors and we show plots of both $p_h$ and $\log|\mathbf{u}_h|$ once the stationary state is reached.

To start with, we consider a full constant permeability tensor
$$\hat{K}(x,y)=\begin{bmatrix}
4 & 0.5 \\
0.5 & 4
\end{bmatrix}.$$
Figure \ref{fig:qfs:full:K}(a) shows the numerical pressure (left) and the logarithm of the norm of the numerical velocity (right), which are symmetric about the diagonal $y=x$, as expected for this homogeneous permeability tensor. In fact, the pressure plot is qualitatively similar to that shown in \cite{cha:pie:for:17}, where an incompressible model is considered. 

In order to check the behaviour of the method when considering a discontinuous full permeability tensor, we consider $\hat{K}$ to be defined as
$$\hat{K}(x,y)=\begin{cases}\,\begin{bmatrix}
16 & 0.5 \\
0.5 & 16
\end{bmatrix},& \hbox{if } x\leq 0.5,\\[3ex]
\,\begin{bmatrix}
4 & 0.5 \\
0.5 & 4
\end{bmatrix},& \hbox{otherwise.}
\end{cases}$$ 
The numerical results are shown in Figure \ref{fig:qfs:full:K}(b). For this  heterogeneous field, the pressure and  the
flow field are no longer symmetric since the fluid seeks to flow in the most high-permeable region.
\begin{figure}[t]
	\begin{center}
		\includegraphics[scale=0.15]{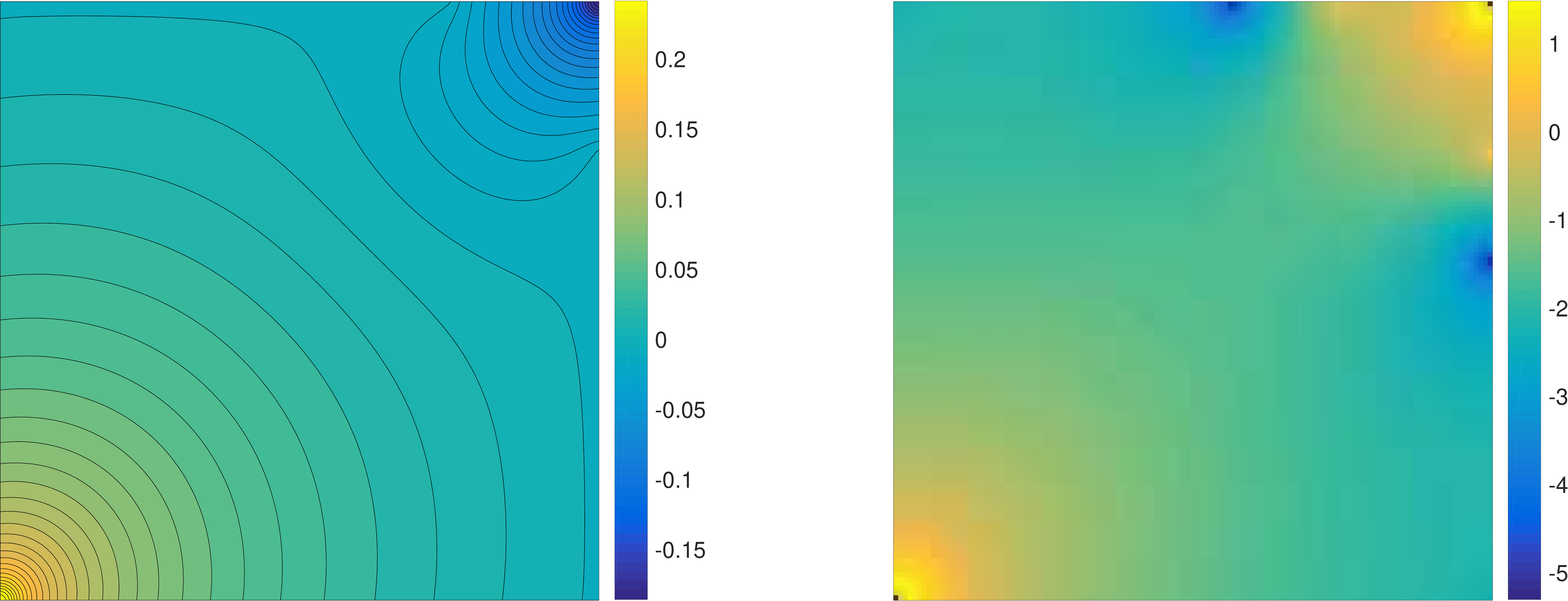}\\[0.2cm]{\footnotesize (a) Full constant permeability tensor}\\[0.7cm]		
		\includegraphics[scale=0.15]{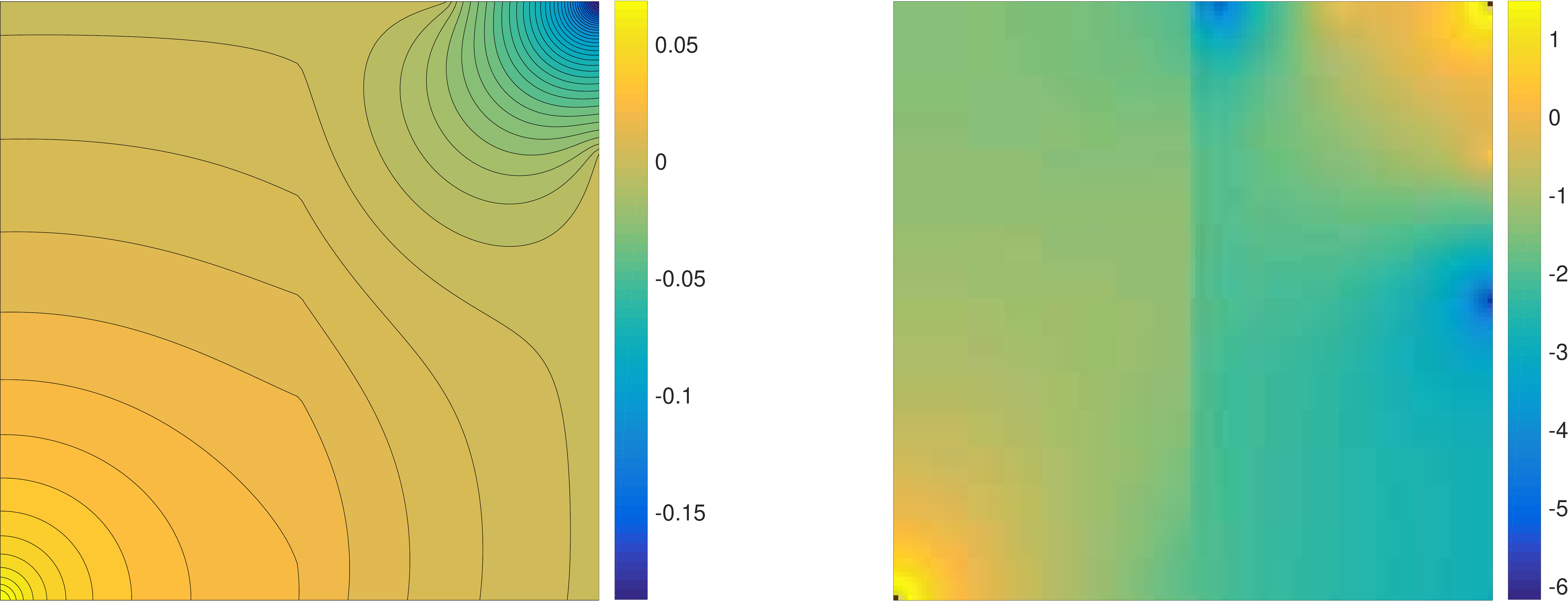}\\[0.2cm]{\footnotesize (b) Full piecewise constant permeability tensor}
	\end{center}\caption{Numerical pressure (left) and logarithm of the norm of the numerical velocity (right) for the quarter five-spot problem considering two full permeability tensors.}\label{fig:qfs:full:K}
\end{figure}
\begin{figure}[t]
	\begin{center}
		\includegraphics[scale=0.15]{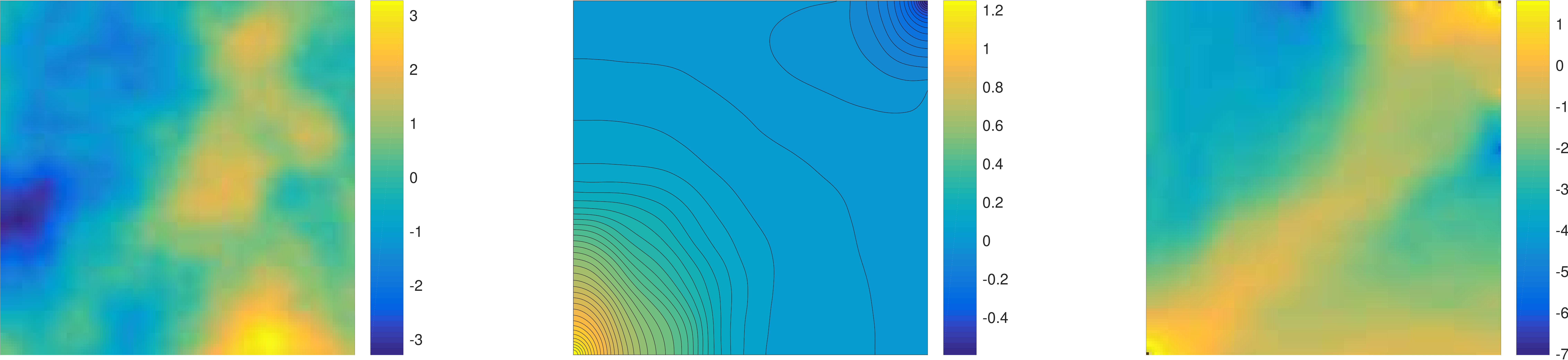}\\[0.2cm]{\footnotesize (a) $(\nu,r,\sigma^2)=(1.5,0.3,1)$}\\[0.7cm]
		\includegraphics[scale=0.15]{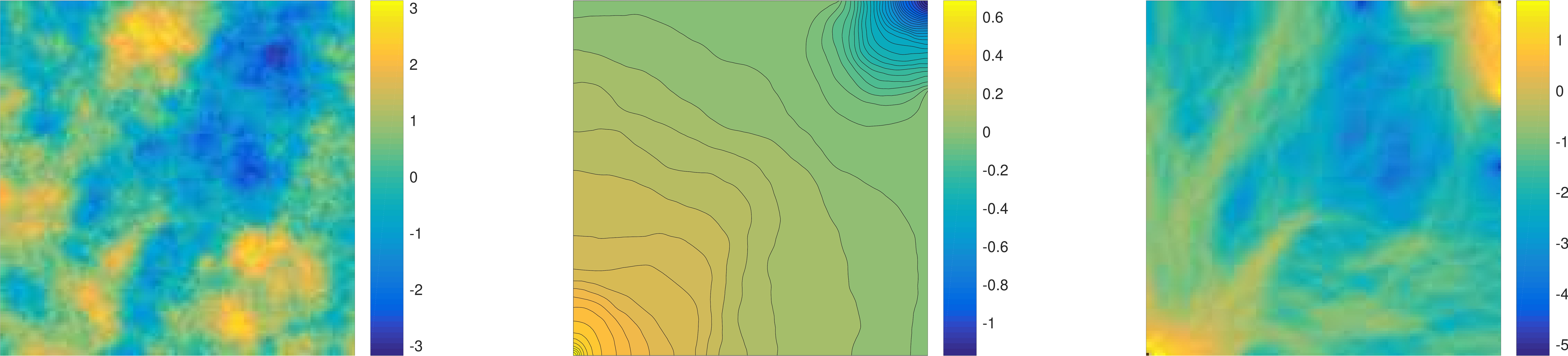}\\[0.2cm]{\footnotesize (b) $(\nu,r,\sigma^2)=(0.5,0.3,1)$}\\[0.7cm]
		\includegraphics[scale=0.15]{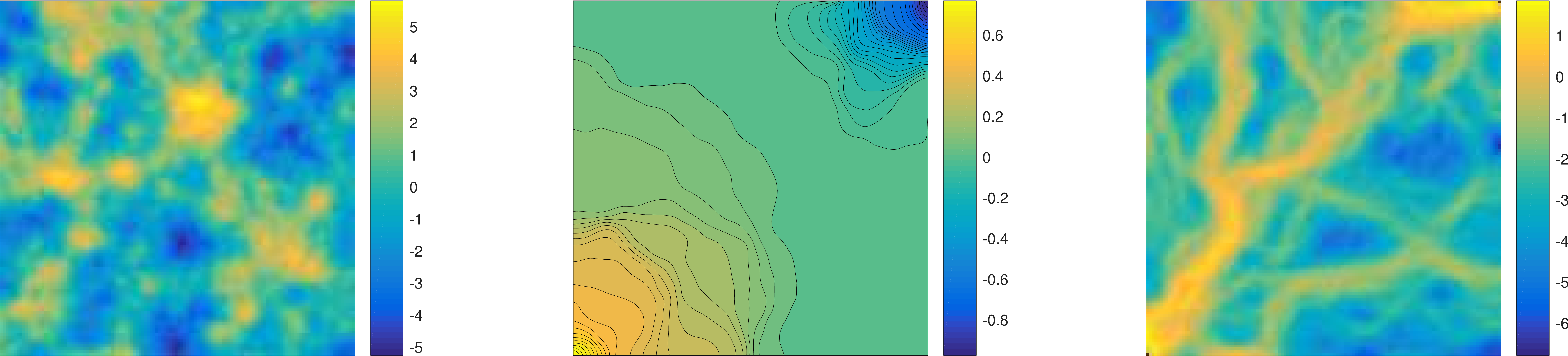}\\[0.2cm]{\footnotesize (c) $(\nu,r,\sigma^2)=(1.5,0.1,3)$}\\[0.7cm]
		\includegraphics[scale=0.15]{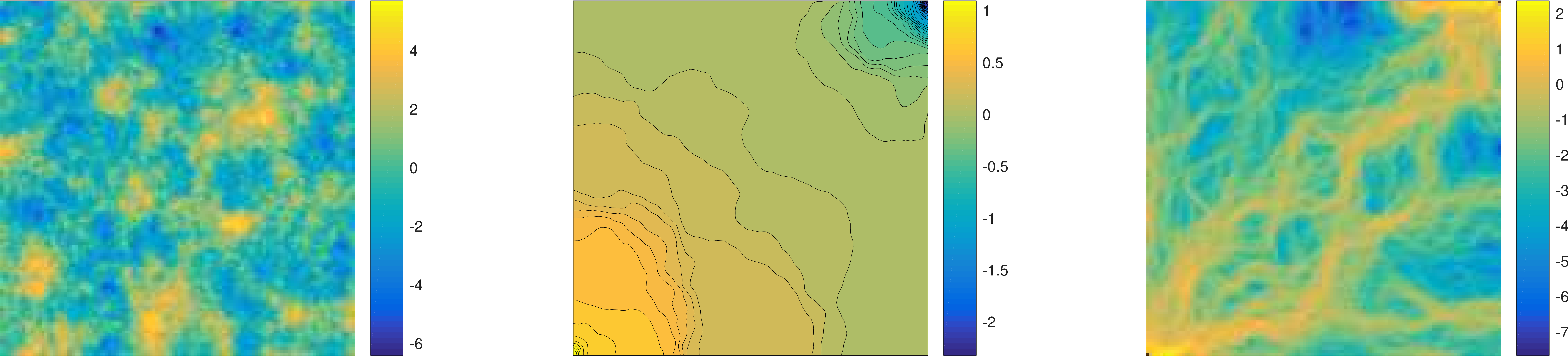}\\[0.2cm]{\footnotesize (d) $(\nu,r,\sigma^2)=(0.5,0.1,3)$}		
	\end{center}\caption{Logarithm of the permeability coefficient (left), numerical pressure (center) and logarithm of the norm of the numerical velocity (right) for the quarter five-spot problem considering four random permeability fields.}\label{fig:qfs:random}
\end{figure}

Finally, we check the behaviour of the method when considering random permeability fields $\hat{K}(x,y)=\hat{k}(x,y)I$. It is well known that  the permeability of a heterogeneous porous medium may be accurately represented by a log-normally distributed random field \cite{fre:75}.  To generate samples of the Gaussian random $\log \hat{k}$, we use the following Mat\'{e}rn-type covariance function \cite{mat:86,min:mcb:05}
$$
C(h)=\frac{\sigma^2}{2^{\nu-1}\Gamma(\nu)}\left(\frac{2\sqrt{\nu}\, h}{r}\right)^{\nu}K_{\nu}\left(\frac{2\sqrt{\nu}\,h}{r}\right),
$$
where $h$ is the separation distance, $K_{\nu}$ is a modified Bessel function of the second kind of order $\nu$, $\Gamma$ is the gamma function, $r$ is the range or distance parameter ($r>0$) which measures how quickly the correlations decay with distance, and $\nu$ is the smoothness parameter ($\nu>0$). One of the main features of the Mat\'{e}rn model is the existence of this parameter $\nu$ which controls the smoothness of the random field, as we shall see below.

In the sequel, we consider four combinations of parameters $\nu$, $r$ and $\sigma^2$, namely $$(\nu,r,\sigma^2)\in\left\{(1.5,0.3,1),(0.5,0.3,1),(1.5,0.1,3),(0.5,0.1,3)\right\}.$$
Figures \ref{fig:qfs:random}(a)-(d) show the logarithm of $\hat{k}$ (left), the pressure (center) and the logarithm of the norm of the velocity (right) for the four sets of parameters stated above. Regarding the variance, Figures \ref{fig:qfs:random}(a), \ref{fig:qfs:random}(b) correspond to permeabilities with small variations, whereas Figures  \ref{fig:qfs:random}(c), \ref{fig:qfs:random}(d) deal with large fluctuations on the permeability fields. Attending to the smoothness, Figures \ref{fig:qfs:random}(a), \ref{fig:qfs:random}(c) consider smooth permeabilities, while Figures \ref{fig:qfs:random}(b), \ref{fig:qfs:random}(d) deal with abrupt permeabilities. In all the cases, the numerical solutions provided by the MFMFE method are satisfactory and show the expected physical behaviour. In this sense, note that, as the smoothness parameter $\nu$ is reduced and/or the variance $\sigma^2$ is increased, the differences between the pressure and velocity fields shown in Figures \ref{fig:qfs:random}(a)-(d) and those for the homogeneous model shown in Figure \ref{fig:qfs:full:K}(a) are wider. Remarkably, the randomness patterns of the permeability coefficient displayed on the left of Figure \ref{fig:qfs:random} are clearly reflected on the corresponding velocity fields shown on the right.

\section*{Acknowledgements}
This work was partially supported by MINECO grant MTM2014-52859-P. The authors gratefully acknowledge Prof. Mary F. Wheeler for suggesting them the subject of this paper.



\bibliography{mybibfile}

\end{document}